\documentclass[preprint,11pt]{elsarticle}

\usepackage[letterpaper,margin=1in]{geometry}
\usepackage{lipsum}
\usepackage{amsfonts}
\usepackage{amsmath}
\usepackage{amsthm}
\usepackage{graphicx}
\usepackage{epstopdf}
\usepackage{tikz,amsmath, amssymb,bm,color}
\usepackage{bm}
\usepackage{caption}
\usepackage[skip=0pt]{subcaption}

\usepackage{algorithm}                                                          
\usepackage{algpseudocode}

\usepackage{lineno,hyperref}
\nolinenumbers
\DeclareMathOperator{\ddiv}{div}
\newcommand{\bmz}[1]{\ensuremath{#1}}
\newtheorem{theorem}{Theorem}
\newtheorem{corollary}{Corollary}[theorem]
\newdefinition{remark}{Remark}
\journal{}

\makeatletter
\def\ps@pprintTitle{%
  \let\@oddhead\@empty
  \let\@evenhead\@empty
  \let\@oddfoot\@empty
  \let\@evenfoot\@oddfoot
}
\makeatother

\begin{document}

\begin{frontmatter}

\title{Oscillation-free numerical schemes for Biot's model and their iterative coupling solution} 
\author[UniZar]{\'Alvaro P\'e de la Riva}\ead{apedelariva@unizar.es}
\address[UniZar]{IUMA and Applied Mathematics Department,
University of Zaragoza, Zaragoza, Spain}
\author[UniZar]{Francisco J. Gaspar}\ead{fjgaspar@unizar.es}
\author[Tufts]{Xiaozhe Hu}\ead{xiaozhe.hu@tufts.edu}
\address[Tufts]{Tufts University, Medford, Massachusetts, USA}
\author[Tufts]{James H. Adler}\ead{james.adler@tufts.edu}
\address[Tufts]{Tufts University, Medford, Massachusetts, USA}
\author[UniZar]{Carmen Rodrigo\corref{cor1}}\ead{carmenr@unizar.es}
\author[NSF]{Ludmil T. Zikatanov}\ead{lzikatan@nsf.gov}
\address[NSF]{U.~S.~National Science Foundation, Alexandria, Virginia, USA}
\cortext[cor1]{Corresponding author. Tel.: +34 976762148; E-mail address: carmenr@unizar.es (C. Rodrigo)}


\begin{abstract}
In this work, we present a new stabilization method aimed at removing spurious oscillations in the pressure approximation of Biot's model for poroelasticity with low permeabilities and/or small time steps. We consider different finite-element discretizations and illustrate how not only does such a stabilized scheme provide numerical solutions that are free of non-physical oscillations, but it also allows one to iterate the fluid and mechanics problems in a fashion similar to the well-known fixed-stress split method. The resulting solution method is convergent without the necessity for additional terms to stabilize the iteration. Finally, we present numerical results illustrating the robust behavior of both the stabilization and iterative solver with respect to the physical and discretization parameters of the model.
\end{abstract}


\end{frontmatter}

\nolinenumbers

\section{Introduction}\label{sec:intro}

The coupling of fluid flow and mechanical deformation within a porous media plays an important role in many relevant applications like geothermal energy extraction, CO2 storage, hydraulic fracturing, and cancer research, among others. 
Such coupling was modeled in the early one-dimensional work of Terzaghi \cite{terzaghi}, whereas the general three-dimensional mathematical formulation was established by Maurice Biot in several pioneering publications \cite{biot1, biot2}. 
There is a vast literature on Biot’s equations and the existence, uniqueness, and regularity of their solutions (see
Showalter~\cite{showalter}, Phillips and Wheeler~\cite{P-W}, and the references therein). 
Nowadays, the analysis and numerical simulation of Biot’s models has become of great interest to scientists and engineers, since reliable numerical methods for solving poroelastic problems are needed for the accurate solution of multi-physics phenomena appearing in a great variety of fields. Biot’s models are used today in different areas ranging
from geomechanics and petroleum engineering, where these models have been applied ever since their discovery, to biomechanics or even food processing more recently. 

Analytical solutions have been derived in the literature for some linear poroelasticity problems (see~\cite{coussybook}), including some that are obtained artificially~\cite{Barry}. Numerical simulations, however, are the only way to obtain quantitative results for real applications and, therefore, the main focus is twofold: first, to choose appropriate discretization schemes that provide numerical solutions whose behavior resembles that of the real problems and, second, to design efficient solvers for the resulting discrete systems. Related to this, it is important to take into account that in many physical applications, the values of some of the parameters appearing in the model may vary over orders of magnitude. For instance, the permeability can typically range from $10^{-9}$ to $10^{-21}$ m$^2$ in geophysical applications~\cite{lee_mardal_winther, wangbook}, and from $10^{-14}$ to $10^{-16}$ m$^2$  in biophysical applications such as in the modeling of soft tissue or bone~\cite{hatira_et_al, Smith_et_al, Mardal_ap_et_al}. Also the Poisson ratio can vary from $0.1$ to $0.5$ in these applications. Due to the large variation of values of these physical parameters, it is important to consider discretizations that are stable, independently of the parameters, and solvers that behave well under such variations. These types of numerical schemes are called parameter-robust methods and have been the focus of the scientific community for the past few years (see for example~\cite{robust_prec, robust_kraus, lee_mardal_winther, preconditioners_sema, New_discret}).

Various discretization schemes for the numerical solution of Biot's model have been developed, including using finite-difference schemes~\cite{Gaspar2003,Gaspar2006} and finite-volume methods (see~\cite{Nordbotten2014,Nordbotten2016} for example). Finite-element methods, which are the subject of this work, have also been widely considered (see for example the monograph by Lewis and Schrefler~\cite{LewisSchrefler} and the references therein and some other more recent works~\cite{jeonghun_lee, lee2023locking, New_discret, son_young_yi_2013}). Problems where the solution is smooth are satisfactorily solved by standard finite-element discretizations. However, when strong pressure gradients appear, it is well-known that these standard schemes often exhibit non-physical oscillations in the numerical approximations to the fluid pressure \cite{Stab_carmen_francisco,Favino2013,ferronato2010,haga_et_al,WheelerPhillip,son_young_yi}. 
This phenomenon may appear at the beginning of the consolidation process when a load is applied on a drain boundary, but also can occur when low permeabilities are considered.
The oscillations can be eliminated by using very fine grids, or when imposing stability restrictions between the space and time discretization parameters (see~\cite{Stab_carmen_francisco, Vermeer_Verruijt}). 
However, such techniques can result in prohibitively expensive computational methods and are not practical.
Instead, the non-physical, oscillatory behavior of the numerical solution can be minimized if approximation spaces for the vector and scalar fields satisfy an appropriate inf-sup condition 
(see~\cite{1974BrezziF-aa}). 
Such discretizations have been analyzed by Murad et al. in~\cite{MuradLoula92, MuradLoula94, MuradLoulaThome}. 
The inf-sup condition, however, does not guarantee numerical solutions free of oscillations. 
This is seen in~\cite{Stab_carmen_francisco}, where Taylor-Hood elements~\cite{Taylor_Hood} 
(continuous piecewise quadratic functions for the approximation of the displacement and 
continuous piecewise linear functions for the pressure) are considered, 
and in~\cite{2016RodrigoGasparHuZikatanov-a} when using the so-called MINI element~\cite{mini}. 
Other numerical schemes to address the oscillatory behavior in the pressure, 
such as least-squares mixed finite-element methods, are proposed in~\cite{korsawe, Tchonkova} for a four-field
formulation of the problem (displacement, stress, fluid flux and pressure unknowns). 
Also, different combinations of continuous and discontinuous Galerkin methods 
and mixed finite-element methods for a three-field formulation are studied in~\cite{P-W, P-W_2, P-W_3}. 
In~\cite{nonconforming}, a nonconforming finite-element method for the three-field formulation of 
Biot’s model is proposed. Piecewise constant elements for the pore pressure paired with the lowest order 
Raviart–Thomas-Nédélec elements for the Darcy’s velocity and the nonconforming Crouzeix–Raviart elements 
for the displacements were considered.  There, the use of mass lumping to eliminate the Darcy velocity
 led to an oscillation-free approximation of the pressure. Additionally, a novel three-field formulation 
 based on displacement, pressure, and total pressure was proposed in~\cite{oyarzua} with error estimates 
 independent of the Lam\'e constants, yielding a locking-free approach. Another strategy to eliminate 
 the non-physical oscillations of the numerical solution of the pressure is to add certain stabilization 
 terms to the Galerkin formulation of the problem. In~\cite{Berger_et_al}, conforming linear finite elements 
 with stabilization for the three-field problem were proposed and analyzed. For the two-field formulation 
 of the problem, this strategy of adding a stabilization term was applied 
 in~\cite{Stab_carmen_francisco, 2016RodrigoGasparHuZikatanov-a} to provide a stable scheme 
 by using linear finite-element approximations for both unknowns (P1-P1) or the classical MINI element scheme. 
 In these latter works, a stabilization by adding the term $\xi h^2 (\nabla \bar{\partial}_t{p_h^n},\nabla q_h)$ to the flow equation, with $\bar{\partial}_t{p_h^n} := (p_h^n - p_h^{n-1})/\tau$, was proposed and theoretically studied. Such a stabilization term, depending on the solid's elastic properties and the mesh size, was given {\it a priori}. Its optimality was shown in the one-dimensional case. This scheme provided solutions without oscillations independently of the discretization parameters, 
but was developed for problems where the storage coefficient $1/\beta$ is close to zero. 

In this work, we propose a novel stabilization, 
which, in the one-dimensional case, is identical to the one presented 
in~\cite{Stab_carmen_francisco, 2016RodrigoGasparHuZikatanov-a} and, in the multi-dimensional, case has several 
advantages. Firstly, the stabilization parameter 
does not depend on the mesh size (which could vary from element to element). 
Secondly, our approach allows alternating iterations between the fluid and the mechanic's problems. 
This is similar to iterative coupling methods, such as the fixed-stress splitting method 
(see for example \cite{Kim2}), but without the additional stabilization needed  
to obtain convergence. 
Moreover, with the aim to obtain parameter-robust methods, 
we are able to extend the approach to the more general case when 
the storage coefficient $1/\beta$ is not close to zero
(see Section~\ref{sec:two_parameters}).

In all the cases above, typical discretization schemes result in a large-scale 
linear system of equations that has to be solved at each time step. 
Such linear systems are usually ill-conditioned and difficult to solve in practice.
 In fact, the solution of the large linear systems of equations arising from the discretization of 
 Biot’s model is the most time-consuming part when real simulations are performed. 
 For this reason, significant effort has been focused on designing efficient iterative solution 
 methods for these problems. There are mainly two 
 approaches that can be adopted: so-called monolithic or fully-coupled methods and iterative coupling methods. 
 Monolithic approaches solve the linear system simultaneously for all the unknowns, 
 and the focus of these is usually in the design of efficient preconditioners to 
 accelerate the convergence of Krylov subspace 
 methods or in the design of appropriate smoothers in a multigrid setting. 
 Recent advances in these directions are found 
 in~\cite{adler_et_al_mg, BERGAMASCHI20072647, ferronato_prec, gaspar_et_al_mg, peiyao}, 
 and the references therein. 
On the other hand, iterative coupling methods solve sequentially the 
equations for fluid flow and geomechanics, at each time step, until a converged solution within a 
prescribed tolerance is achieved. They offer several attractive features, such as their flexibility, 
since two different existing codes can be used for solving the whole coupled poroelastic problem, one for 
the fluid flow and one for the geomechanics.  The most commonly used iterative coupling methods 
are the drained and undrained splittings, which solve the mechanical problem first, and the 
fixed-strain and fixed-stress splittings, which solve the flow problem first~\cite{KIM1, Kim2}. 
The fixed-stress splitting method has been widely 
analyzed~\cite{BORREGALES20191466, BOTH2017101, castelleto_et_al_fs, GASPAR2017526, mikelic_wheeler, storvik}. 
In this work, as a result of the proposed stabilized scheme, 
we obtain a new sequential-implicit method, which is similar to that of the fixed-stress splitting method, 
but with the advantage that no additional stabilization of the iteration is required for convergence.

The structure of the paper is as follows. In Section~\ref{sec:mod_disc}, 
we introduce Biot's model and the discrete schemes considered in this work. 
Section~\ref{sec:stab} describes the new stabilization strategy, demonstrating its uniform stability. 
In Section~\ref{sec:4}, the iterative coupling method arising from the new stabilized scheme is presented, 
and its robust convergence is proved. Section~\ref{sec:1dchoice} motivates the choice of the tuning parameter 
that improves the convergence of the proposed solver, which is done by considering the one-dimensional 
Terzaghi problem. Section~\ref{sec:numerical} demonstrates the good behavior of both the proposed 
stabilization scheme and the corresponding iterative coupling method for two benchmark problems in 
poroelasticity in both two and three dimensions. 
Section~\ref{sec:two_parameters} extends the proposed solution strategy to the more general case in 
which the storage coefficient $1/\beta$ is not close to zero, 
and conclusions are drawn in Section~\ref{sec:conclusions}.

\section{Biot's model: Mathematical model and discretization}\label{sec:mod_disc}

We begin with a brief description of the quasi-static Biot’s model. For a more detailed
explanation about the derivation of the governing equations and the mathematical model, we refer the reader to the books by
Wang~\cite{wangbook} and Coussy~\cite{coussybook}. In addition, we introduce two finite-element schemes that are considered for the spatial discretization of the problem, namely a P1-P1 finite-element scheme and the MINI element discretization.

\subsection{Mathematical model}\label{sec:math_model}

In this work, we consider the quasi-static Biot's model for poroelasticity. We assume that the porous medium is linearly elastic, homogeneous, isotropic, and saturated by a slightly incompressible Newtonian fluid. Under these assumptions, the displacement-pressure formulation for the model is given by
\begin{eqnarray}
& & - \mathbf{\nabla} \left(2\mu \varepsilon(\mathbf{u}) + \lambda \mathbf{\nabla} \cdot \mathbf{u} \right) + \alpha \mathbf{\nabla} p = \mathbf{f}, \label{biot1} \\ 
& &  \frac{1}{\beta} \partial_t p + \partial_t\left(\alpha \mathbf{\nabla} \cdot \mathbf{u} \right) - \nabla \cdot \left(K (\nabla p - \rho_f \mathbf{g}) \right) = S_f, \label{biot2} 
\end{eqnarray}
on a space-time domain $\Omega \times (0,T_f]$, where $\Omega \subset \mathbb{R}^d$, $d \leq 3$, $T_f > 0$. In~\eqref{biot1}-\eqref{biot2}, $\mu$ and $\lambda$ are the Lam\'e parameters, which can be also computed via the Young's modulus $E$ and the Poisson ratio $\nu$ as $\mu =\displaystyle \frac{E}{2(1 + \nu)}$ and $\lambda = \displaystyle\frac{E\nu}{(1-2\nu)(1+\nu)}$. The main variables are $\mathbf{u}$ as the displacement vector and $p$ as the fluid pressure. Here, $\varepsilon(\mathbf{u}) = \frac{1}{2} (\mathbf{\nabla} \mathbf{u} + \mathbf{\nabla}^T \mathbf{u})$ is the linearized strain tensor, $\alpha$ is the Biot cofficient, $\beta$ is the Biot modulus, $\rho_f$ is the fluid density, $\mathbf{g}$ is the gravity tensor,  $K $ is the hydraulic conductivity, and $S_f$ is a source term. 
For simplicity, we assume homogeneous Dirichlet boundary conditions for both displacements and pressure on $\partial \Omega \times (0,T_f]$, and initial conditions $\mathbf{u}(\mathbf{x},0) = \mathbf{u_0} (\mathbf{x})$,
$p(\mathbf{x},0) = p_0 (\mathbf{x})$. 
The existence and uniqueness of the solution for these problems have been investigated by Showalter in~\cite{showalter} and by Zenisek in~\cite{zenisek}.  Finally, we note that in most applications, $1/ \beta$ is small, even very close to zero. While we consider this regime for most of this work, in Section~\ref{sec:two_parameters} we generalize the proposed approach to the case in which $1/\beta$ is not close to zero.

We consider the Sobolev spaces ${\mathbf V} = (H^1_0(\Omega))^d$ and $Q = H^1_0(\Omega)$, where $H^1_0(\Omega)$ denotes the Hilbert subspace of $L^2(\Omega)$ of functions with first weak derivatives in $L^2(\Omega)$ that are zero on the boundary of $\Omega$. We use $(\cdot,\cdot)$ to denote the standard inner product in $L^2(\Omega)$, and define the following bilinear forms,
\begin{eqnarray*}\label{bilinear}
a(\bm{u},\bm{v}) = 2\mu \int_{\Omega}{ \varepsilon}(\bm{u}):{ \varepsilon}(\bm{v}) \, {\rm d} \Omega +
\lambda\int_{\Omega} \ddiv\bm{u}\ddiv\bm{v} \, {\rm d} \Omega, \qquad 
a_p(p,q) =  \int_{\Omega} K \nabla p \cdot \nabla q \, {\rm d} \Omega,
\end{eqnarray*}
corresponding to the elasticity and scaled Laplacian operators, respectively. The variational formulation for the two-field formulation of Biot's model is now: Find $({\mathbf u}(t), p(t))\in {\mathcal C}^1([0,T_f]; {\mathbf V}) \times {\mathcal C}^1([0,T_f];Q)$ such that 
\begin{eqnarray}
  && a(\bm{u},\bm{v}) -\alpha (p, \ddiv \bm{v})  
  = (\bm{f},\bm{v}),
     \quad \forall \  \bm{v} \in \bm V, \label{variational1}\\
  &&  \frac{1}{\beta} (\partial_t{p},q) + \alpha(\ddiv \partial_t{\bm{u}},q)  + a_p(p,q)  = (g,q),
   \; \forall \ q \in Q.\label{variational2}
\end{eqnarray}

\subsection{Discretization}\label{sec:discret}
Let $\mathcal{T}_h$ be a partition of $\Omega\subset \mathbb{R}^d$ consisting of triangles ($d=2)$ or tetrahedrons ($d=3$). For the spatial discretization of Biot's model, we choose
a finite-element pair of spaces ${\bf V}_h \times Q_h$ to approximate the displacements and the pressure, respectively.  For the pressure, we use piecewise linear finite elements, i.e.,
$$
 Q_h = \{ p_h \in H_0^1(\Omega) \, | \, p_h |_T \in P_1, \; \forall T \in \mathcal{T}_h \}.
$$
For the displacements, two different spaces are considered. The first choice, ${\bf V}_h = {\bf V}_l$, consists of the space of piecewise linear continuous vector-valued functions on $\Omega$, i.e.,
$$
{\bf V}_l = \{ {\bm v}_h \in (H_0^1(\Omega))^d \, | \, {\bm v}_h |_T \in (P_1)^d, \; \forall T \in \mathcal{T}_h \}.
$$
The second choice, is the MINI-element~\cite{mini} for which we have ${\bf V}_h = {\bf V}_l \oplus  {\bf V}_b$,  where  ${\bf V}_b$ is the space of element bubble functions.   As a time discretization, we use the backward Euler method on a uniform partition of the time interval $(0, T_f]$, $t_j=j\tau$, $j=0,\ldots, N$, with time-step $\tau=\frac{T_f}{N}$.

For given initial values $({\bf u}_h^0,p_h^0) \in {\bf V}_h \times Q_h$, we have the following fully discrete scheme: for $n=1,2, \ldots, N$, find $({\bf u}_h^n,p_h^n) \in {\bf V}_h \times Q_h$ such that
\begin{eqnarray}
  && a(\bm{u}_h^n,\bm{v}_h) -\alpha (p_h^n, \ddiv \bm{v}_h) 
  = (\bm{f}_h^n,\bm{v}_h),
     \quad \forall \  \bm{v}_h \in \bm V_h, \label{total_discrete_variational1}\\
  &&  \frac{1}{\beta} (\bar{\partial}_t{p_h^n},q_h) + \alpha(\ddiv \bar{\partial}_t{\bm{u}_h^n},q_h)  + a_p(p_h^n,q_h)  = (g_h^n,q_h),
   \; \forall \ q_h \in Q_h,\label{total_discrete_variational2}
\end{eqnarray}
where $\bar{\partial}_t{p_h^n} := (p_h^n - p_h^{n-1})/\tau$ and $\bar{\partial}_t{\bm{u}_h^n} := (\bm{u}_h^n - \bm{u}_h^{n-1})/\tau$. 
We then have the following block form of the discrete problem~\eqref{total_discrete_variational1}-\eqref{total_discrete_variational2}:
\begin{equation}\label{eq:op}
\mathcal{A}\;
\begin{pmatrix}
\bm{u}\\
 p
\end{pmatrix} =
\begin{pmatrix}
\bmz{f} \\
g
\end{pmatrix}, \quad\mbox{with}\quad
\mathcal{A}=\begin{pmatrix}
A& G\\
D &  \tau A_p + \beta^{-1} M
\end{pmatrix},\;
\end{equation}
where operators $A$, $G$, $D$, $A_p$, and $M$ correspond to the
following bilinear forms for $\bm u_h, \bm v_h \in \bm{V}_h$ and $p_h,q_h\in Q_h$:
$$
a(\bm u_h,\bm v_h)\! \rightarrow \! A,  \, 
-\alpha ( p_h, \ddiv \bm v_h) \!  \rightarrow \! G,   \,
\alpha (\ddiv {\bf u}_h, q_h) \!  \rightarrow \! D,  \, a_p( p_h,q_h) \! \rightarrow \! A_p,  \, ( p_h,q_h) \! \rightarrow \! M.
$$
Note that $D = -G^T$ by definition.

\section{On the stability of the proposed method} \label{sec:stab}

The novel stabilization proposed in this work consists of adding two terms to the flow equation and leads to the following discrete variational formulation:
\begin{eqnarray}
  && a(\bm{u}_h^n,\bm{v}_h) -  \alpha(p_h^n,\ddiv \bm{v}_h) = (\bm {f}_h^n,\bm{v}_h),
     \quad \forall \bm{v}_h \in \bm V_h, \label{stab1}\\
  &&  \frac{1}{\beta} (\bar{\partial}_t{p_h^n},q_h) +  \alpha(\ddiv \bar{\partial}_t{\bm{u}_h^n},q_h)  + a_p(p_h^n,q_h)  + L (\bar{\partial}_t{p}_h^n,q_h)_0  - L (\bar{\partial}_t{p}_h^n,q_h) \label{stab2} \\  
  && \quad = (g_h^n,q_h), \forall q_h \in Q_h, \nonumber
\end{eqnarray}
where $(\cdot,\cdot)_0$ is an approximation of the $L^2(\Omega)$ inner product defined by mass lumping, {\it i.e.}, for continuous functions $p$ and $q$ defined on $\bar{\Omega}$, we have
$$
(p,q)_0 = \sum_{T\in\mathcal{T}_h} \int_T(pq)_I\,d\bm{x} = \sum_{T\in\mathcal{T}_h} \frac{|T|}{d+1}\sum_{j=1}^{d+1} (pq)(P_{T,j}), 
$$
where $(pq)_I$ denotes the linear interpolant of the continuous function $(pq)$ and $P_{T,j}$ are the coordinates of the $j$-th vertex of $T\in\mathcal{T}_h$.
In \eqref{stab2}, $L$ is a parameter appropriately chosen to remove the non-physical oscillations. We elaborate on its choice later, depending on the discretization considered for the displacements. 
We now have the following block
form of the discrete problem~\eqref{stab1}-\eqref{stab2}:
\begin{equation*}
\mathcal{A}_{stab}\;
\begin{pmatrix}
\bm{u}\\
 p
\end{pmatrix} =
\begin{pmatrix}
\bmz{f} \\
g
\end{pmatrix}, \quad\mbox{with}\quad
\mathcal{A}_{stab}=\begin{pmatrix}
A& G\\
D &  C 
\end{pmatrix}\;,
\end{equation*}
where $C = \tau A_p + \beta^{-1} M + L(M_l- M)$, with $M$ and $M_l$ the mass matrix and the lumped mass matrix, respectively.

To prove the new scheme's stability, we show a relationship between the technique introduced in~\cite{2016RodrigoGasparHuZikatanov-a} and the proposed stabilization, which is defined by the difference between the lumped mass matrix $M_l$ and the mass matrix $M$ on a simplicial grid. By a straightforward calculation, we have the following conditions for piecewise linear continuous functions $p$ and $q$,
\begin{eqnarray*}
   && \langle M p,q\rangle = \sum_{T\in\mathcal{T}_h} \int_Tpq\,d\bm{x},\quad 
    \langle M_l p,q\rangle = \sum_{T\in \mathcal{T}_h} \int_T(pq)_I\,d\bm{x}, \\
   && \langle Zp,q \rangle:= \langle (M_l - M) p,q\rangle = \sum_{T\in\mathcal{T}_h} \langle Z_T p_T,q_T\rangle.
\end{eqnarray*}
For a fixed $T\in\mathcal{T}_h$, $\langle Z_Tp_T,q_T\rangle$ is the difference between the integrals in the definitions of $M_l$ and $M$ and $p_T \in \mathbb{R}^{d+1}$ is the vector that represents the degrees of freedom of $p(x)$ with components $\{p_{T,k}\}_{k=1}^{d+1}$. Evaluating the integrals on $T\in \mathcal{T}_h$ and using the Poincar\'e inequality for convex domains~\cite[Theorem 3.2]{2003Bebendorf-a}, we have that
\begin{equation*}
  \frac{1}{d+1}\langle Z_T p_T,p_T\rangle = \int_T\left(p-\frac{1}{|T|}\int_{T} p\,d\bm{x}\right)^2\,d\bm{x}\le \frac{h_T^2}{\pi^2}\int_T|\nabla p|^2\,d\bm{x}.
\end{equation*}
On the other hand, from \cite[Equations~(2.3)-(2.4)]{1999XuZikatanov-a}, we obtain the bound
\begin{equation}
  h_T^2\int_T |\nabla p|^2\,d\bm{x} \le
  \frac{c_T|T|}{2}\sum_{j=1}^{d+1}\sum_{k=1}^{d+1}(p_{T,j}-p_{T,k})^2=c_T(d+1)(d+2)\langle Z_T p_T,p_T\rangle,
  \label{ineq2}
\end{equation}
where $c_T :=\displaystyle\max_{\substack{1\le j,k \le(d+1)\\ j\neq k}}\left\{\frac{h_T^2}{|h_j||h_k|}\right\}$. 
Summing over all elements $T\in \mathcal{T}_h$ shows that $Z = (M_l-M)$ is spectrally equivalent to the scaled stiffness matrix $h^2 L_p$ (here we define $\langle h^2 L_p p, q \rangle := \sum_{T \in \mathcal{T}} h_T^2 \int_{T} |\nabla p|^2\,d\bm{x}$), i.e., 
\begin{equation}\label{ine:spec-equiv}
C_1 \langle Z p, q \rangle \leq \langle h^2L_p p, q \rangle \leq C_2 \langle Z p, q \rangle.
\end{equation}
Hence, an inf-sup condition similar to the one shown in~\cite[Theorem 1]{2016RodrigoGasparHuZikatanov-a} holds.  

In the one-dimensional case, inequality \eqref{ineq2} turns into an equality, and, hence, $h^2 L_p = 6(M_l -M)$. In~\cite{2016RodrigoGasparHuZikatanov-a}, the minimum stabilization providing monotone discretizations in the one-dimensional case was shown to be $\xi h^2 L_p$, with $\xi = \alpha^2/(4(\lambda+2\mu))$ for P1-P1 and $\xi = \alpha^2/(6(\lambda+2\mu))$ for the MINI element. Thus, choosing parameter $L$ in \eqref{stab2} as 
$$L = \displaystyle \left\{ \begin{array}{ll} \displaystyle \frac{3\alpha^2}{2(\lambda+2\mu/d)}, & \hbox{for P1-P1,} \\  \displaystyle \frac{\alpha^2}{\lambda+2\mu/d}, & \hbox{for MINI-element,}  \end{array}\right.$$
we then obtain a new stabilized scheme that is identical, as expected, to the one given in~\cite{2016RodrigoGasparHuZikatanov-a} for one-dimensional problems. For two and three spatial dimensions, it was observed numerically in~\cite{2016RodrigoGasparHuZikatanov-a} that adding the term  $\xi h^2 L_p$  to the stabilization leads to monotone schemes, though no theoretical results exist. 
As we noted earlier, the stabilization parameter $L$ does not depend on the mesh size. In addition, as it is shown in next section, our scheme also gives a novel fixed-stress splitting method. 

\section{On the convergence of the iterative coupling method}
\label{sec:4}
A popular approach for iteratively solving Biot's equations is first to solve the flow problem and then the mechanics, and iterate until a converged solution is obtained. This results in the so-called fixed-stress splitting method. In order to achieve convergence of the fixed-stress iteration, the artificial term, $L (\bar{\partial}_t{p}_h^n,q_h)$, has to be added to both sides of the flow equation, choosing parameter $L$  large enough. In our case, similar terms have already been added to the flow equation in the discrete problem, \eqref{stab1}-\eqref{stab2}, aiming to remove the unphysical oscillations.  We show that this is enough to ensure convergence of the resulting iterative strategy.
However, to optimize convergence of the scheme, we introduce a new parameter $\gamma > 0$ in the flow equation. 
Therefore, at each time step, $t_n$, we propose the following iterative method for the solution of \eqref{stab1}-\eqref{stab2}.  Given the initial guess for the coupling iteration at time $t_n$, ${\bm u}_h^{n,0} = {\bm u}_h^{n-1}$ and $p_h^{n,0} = p_h^{n-1}$, the algorithm provides a sequence of approximations $({\bm u}_h^{n,i}$ and $p_h^{n,i}), i \geq 1$, as follows: \\
\noindent {\bf Step 1:} Given $({\bm u}_h^{n,i-1},p_h^{n,i-1}) \in {\bm V}_h \times Q_h$, find $p_h^{n,i} \in Q_h$ such that:
\begin{eqnarray}  
  &&  \frac{1}{\beta} \left(\frac{p_h^{n,i}-p_h^{n-1}}{\tau},q_h\right) + \gamma L \left(\frac{p_h^{n,i}-p_h^{n-1}}{\tau},q_h\right)_0 + a_p(p_h^{n,i},q_h) \label{it_pressure2}\\
  && = -\alpha \left(\ddiv \frac{\bm{u}_h^{n,i-1}-\bm{u}_h^{n-1}}{\tau},q_h\right) +  L \left(\frac{p_h^{n,i-1}-p_h^{n-1}}{\tau},q_h\right)  \nonumber  \\ 
  && \qquad + (\gamma-1) L \left(\frac{p_h^{n,i-1}-p_h^{n-1}}{\tau},q_h\right)_0 +  (g_h^n,q_h), \quad \forall \ q_h \in Q_h, \nonumber 
\end{eqnarray}
{\bf Step 2:} Given $p_h^{n,i} \in Q_h$, find ${\bm u}_h^{n,i} \in {\bm V}_h$ such that
\begin{equation}
a(\bm{u}_h^{n,i},\bm{v}_h) =  \alpha( p_h^{n,i}, \ddiv \bm{v}_h) + (\bm{f}_h^n,\bm{v}_h),
     \quad \forall \  \bm{v}_h \in \bm V_h. \label{it_displacement2}
\end{equation}

\noindent In matrix form, the iterative method is based on the following splitting of $\mathcal{A}_{stab}$,
\begin{equation}\label{splitting_gamma}
\mathcal{A}_{stab} = 
\begin{pmatrix}
A& G\\
0 &  \tau A_p + \beta^{-1} M + \gamma L M_l
\end{pmatrix}\;
-\begin{pmatrix}
0& 0\\
-D &  L M + (\gamma-1) L M_l
\end{pmatrix}.
\end{equation}

Next, we prove the convergence of the iterative method for a set of values of parameter $\gamma$.
In particular, we demonstrate that algorithm \eqref{it_pressure2}-\eqref{it_displacement2}  is convergent for any parameter $\gamma \in (1/2,2]$ and $L$ such that $\gamma L$ is big enough. Additionally, the constants arising from the analysis are independent of any physical and discretization parameters, showing the robustness of the proposed iterative scheme. 

Let $e_u^i = \bm{u}_h^{n,i} - \bm{u}_h^n$ and $e_p^i = p_h^{n,i}-p_h^n$ denote the errors at iteration $i$  for the displacements and for the pressure, respectively.  Moveover, we assume the following weak inf-sup condition,
\begin{equation}\label{ine:weak-inf-sup}
	\sup_{\bm{w}_h \in \bm{V}_h} \frac{(q, \ddiv \bm{w}_h)}{\| \bm{w}_h  \|_A} \geq \eta \frac{1}{\sqrt{\lambda + 2\mu/d}} \| q \| - \epsilon \frac{1}{\sqrt{\lambda + 2\mu/d}} h \| \nabla q \|,   \quad \forall \ q \in Q_h,
\end{equation}
where $\eta >0$ and $\epsilon > 0$ are constants that do not depend on the mesh size or the physical parameters.  Both P1-P1 and the MINI-element satisfy~\eqref{ine:weak-inf-sup} (see \cite{2016RodrigoGasparHuZikatanov-a,adler2018robust,preconditioners_sema}), with the MINI-element satisfying a stronger version with $\epsilon = 0$. 
\begin{theorem} \label{thm:robust-convergence-one-gamma}
If the weak inf-sup condition \eqref{ine:weak-inf-sup} holds, the iterative method given in \eqref{it_pressure2}-\eqref{it_displacement2} 
is convergent for any parameters $\gamma \in (1/2,2]$ and $L$ such that  $\gamma L = \omega \frac{\alpha^2}{ (\lambda+2\mu/d)} \geq \frac{\alpha^2}{(\lambda+2\mu/d)}$, i.e., $\omega \geq 1$. Additionally,
\begin{equation}
\label{thm_contraction_pressure_3}
\| e_p^i \|^2  + \frac{|1-\gamma|}{\gamma} \| e_p^i \|_Z^2   \leq  \frac{1}{ 1+  \frac{\eta^2}{\omega (1 + 2  \theta^*) } } \left(  \| e_p^i \|^2 + \frac{|1-\gamma|}{\gamma}  \| e_p^{i-1} \|_Z^2  \right),
\end{equation}
where $\theta^* \geq \frac{\epsilon^2C_2 \gamma}{4 \omega(\gamma - |1-\gamma|)}$ is a root of the following quadratic equation
\begin{equation} \label{def:q2}
	q_2(\theta):=
	\left[ \frac{2(\gamma-|1-\gamma|)}{\gamma}  \right] \theta^2 + \left[ \frac{\gamma-|1-\gamma|}{\gamma} - \frac{\eta^2 |1-\gamma|}{2 \omega \gamma} - \frac{\epsilon^2 C_2}{2\omega} \right] \theta - \frac{\epsilon^2 C_2}{4 \omega } = 0.
\end{equation}
Here, $\eta > 0$ is the constant appearing in the weak inf-sup condition~\eqref{ine:weak-inf-sup}, and $C_2$ is the constant for the upper bound of the spectral equivalence condition~\eqref{ine:spec-equiv}.
\end{theorem}
\begin{proof}
By taking the differences of corresponding equations \eqref{stab1}-\eqref{stab2} and equations \eqref{it_pressure2}-\eqref{it_displacement2}, testing with $\bm{v}_h = e_{\bm{u}}^{i-1} \in V_h$
and $q_h = e_p^i \in Q_h$, and adding all together, we obtain
\begin{equation}
\label{first_equality_3}
a( e_{\bm{u}}^i,  e_{\bm{u}}^{i-1}) + \frac{1}{\beta} \| e_p^i \|^2 + \tau \|e_p^i\|_{A_p}^2 + \gamma L(e_p^i,e_p^i)_0 +(1 - \gamma) L (e_p^{i-1},e_p^i)_0  - L(e_p^{i-1},e_p^i) = 0.
\end{equation}
Using $M_l-M=Z$, 
\begin{align}
\label{first_inequality_31}
a( e_{\bm{u}}^i,  e_{\bm{u}}^{i-1}) + \frac{1}{\beta} \| e_p^i \|^2 + \tau \|e_p^i\|_{A_p}^2 +&\\\nonumber
+ \gamma L(e_p^i-e_p^{i-1},e_p^i) &+ \gamma L(e_p^i,e_p^i)_Z  +(1 - \gamma) L (e_p^{i-1},e_p^i)_Z   = 0.
\end{align}
Using a polarization identity and Young's inequality, we have
\begin{align}
&\quad \frac{1}{2} \| e_{\bm{u}}^2 \|_A^2 + \frac{1}{2} \| e_{\bm{u}}^{i-1} \|_A^2 + \frac{1}{\beta} \| e_p^i \|^2+ \tau \| e_{p}^i \|^2_{A_p} + \frac{\gamma L}{2} \| e_p^i \|^2 + \frac{\gamma L}{2} \| e_p^i - e_p^{i-1} \|^2  + \nonumber  \\
& + \left( \gamma - \frac{|1-\gamma|}{2}  \right) L \| e_p^i \|_Z^2\leq \frac{\gamma L}{2} \|e_p^{i-1} \|^2 + \frac{1}{2} \| e_{\bm{u}}^i - e_{\bm{u}}^{i-1} \|_A^2 + \frac{|1-\gamma|}{2} L \| e_p^{i-1} \|_Z^2. \label{ine:before-sub}
\end{align}
Taking the difference between equations \eqref{it_displacement2}  and \eqref{stab1} evaluated at iteration $i$ and $i-1$, tested with $\bm{v}_h = e_{\bm{u}}^i - e_{\bm{u}}^{i-1}$ yields
\begin{align*}
	\| e_{\bm{u}}^i  - e_{\bm{u}}^{i-1} \|_A^2 & = \alpha (e_p^i -e_p^{i-1}, \nabla \cdot ( e_{\bm{u}}^i  - e_{\bm{u}}^{i-1})) \leq \alpha \| e_p^i -e_p^{i-1} \| \| \nabla \cdot ( e_{\bm{u}}^i  - e_{\bm{u}}^{i-1})\|  \\
	& \leq \frac{\alpha}{\sqrt{\lambda + 2\mu/d}} \| e_p^i -e_p^{i-1} \| \| e_{\bm{u}}^i  - e_{\bm{u}}^{i-1} \|_A, 
\end{align*}
where we use the Cauchy-Schwarz inequality and the fact that $a(\bm{u}, \bm{u}) \geq (\lambda + \frac{2\mu}{d}) \| \nabla \cdot \bm{u} \|^2$.  Thus, 
\begin{equation} \label{ine:mech_ineq3}
	\| e_{\bm{u}}^i  - e_{\bm{u}}^{i-1} \|_A \leq \frac{\alpha}{\sqrt{\lambda + 2\mu/d}} \| e_p^i -e_p^{i-1} \|.
\end{equation}
From~\eqref{ine:weak-inf-sup}, we have, for any given $e_p^i$, there exists $\bm{w}_h \in V_h$, such that
\begin{equation*}
	(e_p^i, \ddiv \bm{w}_h) \geq \left( \eta \frac{1}{\sqrt{\lambda + 2\mu/d}} \| e_p^i \| - \epsilon \frac{1}{\sqrt{\lambda + 2\mu/d}} h \| \nabla e_p^i \|  \right) \| \bm{w}_h \|_A, \quad \| \bm{w}_h \|_A = \| e_p^i \|.
\end{equation*}
Taking the difference between equations \eqref{it_displacement2}  and \eqref{stab1} and testing with $\bm{v}_h = \bm{w}_h$,
\begin{align*}
	& \quad \alpha \left( \eta \frac{1}{\sqrt{\lambda + 2\mu/d}} \| e_p^i \| - \epsilon \frac{1}{\sqrt{\lambda + 2\mu/d}} h \| \nabla e_p^i \|  \right) \| \bm{w}_h \|_A \leq \alpha (e_p^i,  \ddiv \bm{w}_h) \\
	&= a(e_{\bm{u}}^i, \bm{w}_h) \leq \| e_{\bm{u}}^i \|_A \| \bm{w}_h \|_A,
\end{align*}
which implies 
\begin{equation*}
		\| e_{\bm{u}}^i \|_A \geq \eta \frac{\alpha}{\sqrt{\lambda + 2\mu/d}} \| e_p^i \| - \epsilon \frac{\alpha}{\sqrt{\lambda + 2\mu/d}} h \| \nabla e_p^i \|.
	\end{equation*}
Using Young's inequality with constant $\theta > 0$, we have
\begin{equation} \label{ine:low-bound_eu_A_norm}
		\| e_{\bm{u}}^i \|^2_A \geq \frac{\eta^2}{1+2\theta} \frac{\alpha^2}{\lambda + 2\mu/d} \| e_p^i \|^2 - \frac{ \epsilon^2}{2\theta} \frac{\alpha^2}{\lambda + 2\mu/d} h^2 \| \nabla e_p^i \|^2.
\end{equation}
Substituting~\eqref{ine:mech_ineq3} and~\eqref{ine:low-bound_eu_A_norm} back into \eqref{ine:before-sub}, dropping the $\frac{1}{2} \| e_{\bm{u}}^{i-1} \|_A^2$, $\frac{1}{\beta} \| e_p^i \|^2$, and $\tau \| e_p^i \|^2_{A_p} $ terms, we arrive at
\begin{align*}
& \quad  \left( \frac{1}{2} \frac{\eta^2}{1 + 2 \theta} \frac{\alpha^2}{\lambda + 2\mu/d} + \frac{\gamma L}{2} \right) \| e_p^i \|^2 + \frac{\gamma L}{2} \| e_p^i - e_p^{i-1} \|^2 + \left( \gamma - \frac{|1-\gamma|}{2 }  \right) L \| e_p^i \|_Z^2 \\
& \leq \frac{\gamma L}{2} \| e_p^{i-1} \|^2 + \frac{1}{2} \frac{\alpha^2 }{\lambda + 2\mu/d} \|e_p^i -e_p^{i-1} \|^2 + \frac{|1-\gamma|}{2} L \| e_p^{i-1} \|_Z^2 + \frac{\epsilon^2 C_2}{4 \theta} \frac{\alpha^2}{\lambda + 2\mu/d} \| e_p^i \|_Z^2.
\end{align*}
If $\gamma L = \omega  \frac{\alpha^2}{\lambda + 2\mu/d} \geq \frac{\alpha^2}{\lambda + 2\mu/d}$, i.e., $\omega \geq 1$, we have
\begin{align*}
& \quad \left( \frac{1}{2} \frac{\eta^2}{1 + 2 \theta} \frac{\gamma}{\omega}  + \frac{\gamma }{2} \right) \| e_p^i \|^2 +\left( (1 - \frac{\epsilon^2 C_2}{4 \omega \theta}) \gamma - \frac{|1-\gamma|}{2 }  \right)  \| e_p^i \|_Z^2 \\
& \leq \frac{\gamma }{2} \| e_p^{i-1} \|^2 + \frac{|1-\gamma|}{2} \| e_p^{i-1} \|_Z^2.
\end{align*}
To ensure convergence, we require that,
\begin{equation*}
(1 - \frac{\epsilon^2 C_2}{4 \omega \theta}) \gamma - \frac{|1-\gamma|}{2}   \geq \frac{|1-\gamma|}{2} \Longrightarrow \theta \geq \frac{\epsilon^2 C_2 \gamma}{4\omega (\gamma - |1-\gamma|)}.
\end{equation*}
Next, we try to find $\theta = \theta^* \geq \frac{\epsilon^2 C_2 \gamma}{4 \omega (\gamma-|1-\gamma|)} $ such that
\begin{equation*}
	\frac{(1 - \frac{\epsilon^2 C_2}{4 \omega \theta}) \gamma - \frac{|1-\gamma|}{2 } }{ \frac{1}{2} \frac{\eta^2}{1 + 2 \theta} \frac{\gamma}{\omega}  + \frac{\gamma }{2}} = \frac{|1-\gamma|}{\gamma}.
\end{equation*}
Direct calculations show that $\theta^* \geq \frac{\epsilon^2 C_2 \gamma}{4 \omega (\gamma -|1-\gamma|)}$ should be the positive root of the  quadratic equation~\eqref{def:q2}.  Since $\frac{2(\gamma - |1-\gamma|)}{\gamma} > 0$ for $\gamma \in (\frac{1}{2}, 2]$, the existence of $\theta^*$ is verified by the positiveness of the discriminant,  and the fact that $q_2(\frac{\epsilon^2 C_2 \gamma}{4 \omega (\gamma-|1-\gamma|)}) \leq 0$.  Therefore, we have
\begin{equation*}
 \left( \frac{\eta^2 \gamma}{2 \omega (1 + 2 \theta^*) } + \frac{\gamma }{2} \right) \left( \| e_p^i \|^2  + \frac{|1-\gamma|}{\gamma} \| e_p^i \|_Z^2 \right)  \leq  \frac{\gamma }{2} \left(  \| e_p^i \|^2 + \frac{|1-\gamma|}{\gamma}  \| e_p^{i-1} \|_Z^2  \right),
\end{equation*}
which completes the proof. 
\end{proof}

Finally, we give a theorem which determines how to choose parameter $\gamma$ so that the proposed iterative coupling method is optimal, that is, it converges in only two iterations. 

\begin{theorem}\label{thm:optimal}
Given the system of equations ${\cal A} x = b$, where matrix A has the form
$\mathcal{A} = 
\begin{pmatrix}
	A & G \\
	D & C
\end{pmatrix}$,
the iterative method  based on the following splitting of $\mathcal{A}$,
\begin{equation}\label{splitting}
\begin{pmatrix}
A& G\\
D &   C
\end{pmatrix}\; = 
\begin{pmatrix}
A& G\\
0 &  S_p
\end{pmatrix}\;
	-\begin{pmatrix}
		0& 0\\
		-D & S_p-C
	\end{pmatrix},
\end{equation}
where $S_p$ is the Schur complement $S_p = C - DA^{-1}G$, converges in two iterations.
\end{theorem}
\begin{proof}
Denote by $\mathcal{S}$ the iteration matrix corresponding to splitting \eqref{splitting}, i.e., $\mathcal{S} = \mathcal{I} - \mathcal{B}\mathcal{A}$ where $\mathcal{B} = 
\begin{pmatrix}
	A & G \\
	0 & S_p
\end{pmatrix}^{-1}$.
Note that $$\mathcal{S}^2 = (\mathcal{I} - \mathcal{B}\mathcal{A})( \mathcal{I} - \mathcal{B}\mathcal{A}) = \mathcal{B} (\mathcal{I} - \mathcal{A}\mathcal{B}) (\mathcal{I} - \mathcal{A}\mathcal{B}) \mathcal{B}^{-1}.$$
Therefore, if $(\mathcal{I} - \mathcal{A}\mathcal{B})^2 = 0$, then $\mathcal{S}^2 = 0$. 
We now consider the block-$\mathcal{L} \mathcal{U}$ factorization of $\mathcal{A}$, i.e.,
\begin{equation*}
\mathcal{A} =
\begin{pmatrix}
	A & G \\
	D & C 
\end{pmatrix}
=
\begin{pmatrix}
I & 0 \\
D A^{-1} & I 
\end{pmatrix}
\begin{pmatrix}
	A & G \\
	0 & S_p 
\end{pmatrix}
=: \mathcal{L} \mathcal{U}.
\end{equation*}
Since  $\mathcal{B} = \mathcal{U}^{-1}$, then $\mathcal{I} - \mathcal{A}\mathcal{B} = \mathcal{I} - \mathcal{L}\mathcal{U} \mathcal{U}^{-1} = \mathcal{I} - \mathcal{L} = 
\begin{pmatrix}
	0 & 0 \\
	-D A^{-1} & 0
\end{pmatrix}$.
This implies that $(\mathcal{I} - \mathcal{A}\mathcal{B})^2 = 0$ and, consequently, $\mathcal{S}^2 = 0$. 
\end{proof}

In order to use Theorem~\ref{thm:optimal} to choose the optimal value of parameter $\gamma$ for the proposed iterative method, we need to be able to compute the Schur complement $S_p$ so that the splitting \eqref{splitting_gamma} of matrix  $\mathcal{A}_{stab}$ defining our method has the form given in~\eqref{splitting}. In order to do that, in the next section, we consider the one-dimensional Terzaghi's problem for poroelasticity, for which we can explicitly obtain $S_p$ for the two stabilized discretizations proposed. 

\section{One-dimensional problem and choice of parameter $\gamma$}\label{sec:1dchoice}
In this section, we determine an optimal value of the parameter $\gamma$ by studying a one-dimensional problem, namely the well-known Terzaghi's problem, where the Schur complement, $S_p = C-D A^{-1}G$, can be calculated explicitly.  
The Terzaghi problem models a column of a porous medium with height, $H$,
saturated by an incompressible fluid, bounded by impermeable and rigid
lateral walls and bottom, and supporting a load $\sigma_0$ on the top boundary
which is free to drain. This results in the following one-dimensional version of model \eqref{biot1}-\eqref{biot2}:
\begin{equation} \label{example-int}
\begin{array}{l}
-\displaystyle \frac{\partial}{\partial x} \left( (\lambda+2\mu)  \, \frac{\partial u}{\partial x}\right) + \displaystyle \frac{\partial p}{\partial x}=0, \\
\displaystyle \frac{\partial }{\partial t}\left (\displaystyle \frac{\partial u}{\partial
x}\right)-\displaystyle \frac{\partial}{\partial x} \left( K \, \frac{\partial p}{\partial x}\right) = 0,
\end{array}
(x,t)\in (0,H)\times(0,T_f],
\end{equation}
where for simplicity we have fixed $\alpha=1$ and $1/\beta=0$. This system is supplemented with the following boundary and initial conditions,
$$
\begin{array}{l}
\displaystyle (\lambda+2\mu)  \, \frac{\partial u}{\partial x}(0,t)=\sigma_0, \quad p(0,t)=0, \ t\in (0,T_f], \\
u(H,t)=0,  \ \displaystyle K \frac{\partial p}{\partial x}(H,t)=0, \ t\in (0,T_f], \\
\displaystyle \frac{\partial u}{\partial x}(x,0)=0, \ x\in [0,H].
\end{array}
$$
A uniform partition of spatial domain $\Omega=(0,H)$ with mesh size $h$ is considered, and 
the backward Euler method is chosen for discretization in time. Two spatial discretizations are considered, namely the stabilized P1-P1 and the stabilized MINI-element method.

\subsection{Discretization with linear finite elements (P1-P1)}
First, we discretize using linear finite elements for both displacement and pressure. Fluid pressure oscillations are observed even for this simple problem. In Figure~\ref{figuras_P1_stabilization}(a), we show the numerical solution obtained 
for the pressure field at the final time $T_f=0.1$, taking the hydraulic conductivity $K=10^{-6}$, $\lambda+2\mu =1$, Biot coefficient $\alpha=1$, and a mesh size $h=1/32$. Adding the proposed stabilization with parameter $\displaystyle L=\frac{3\alpha^2}{2(\lambda+2\mu)}$, however, eliminates the oscillations, as shown in Figure~\ref{figuras_P1_stabilization}(b). 
\begin{figure}[h!]
\begin{tabular}{cc}
\includegraphics[width = 0.45\textwidth]{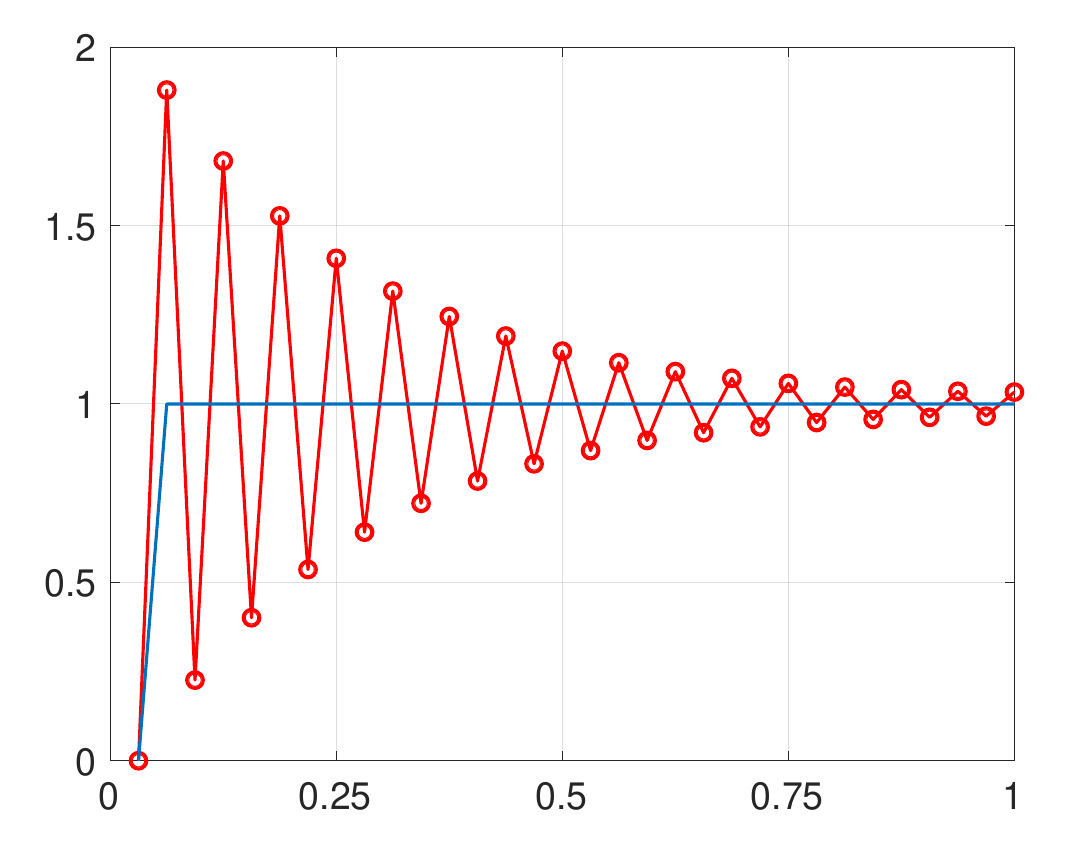}
&
\includegraphics[width = 0.45\textwidth]{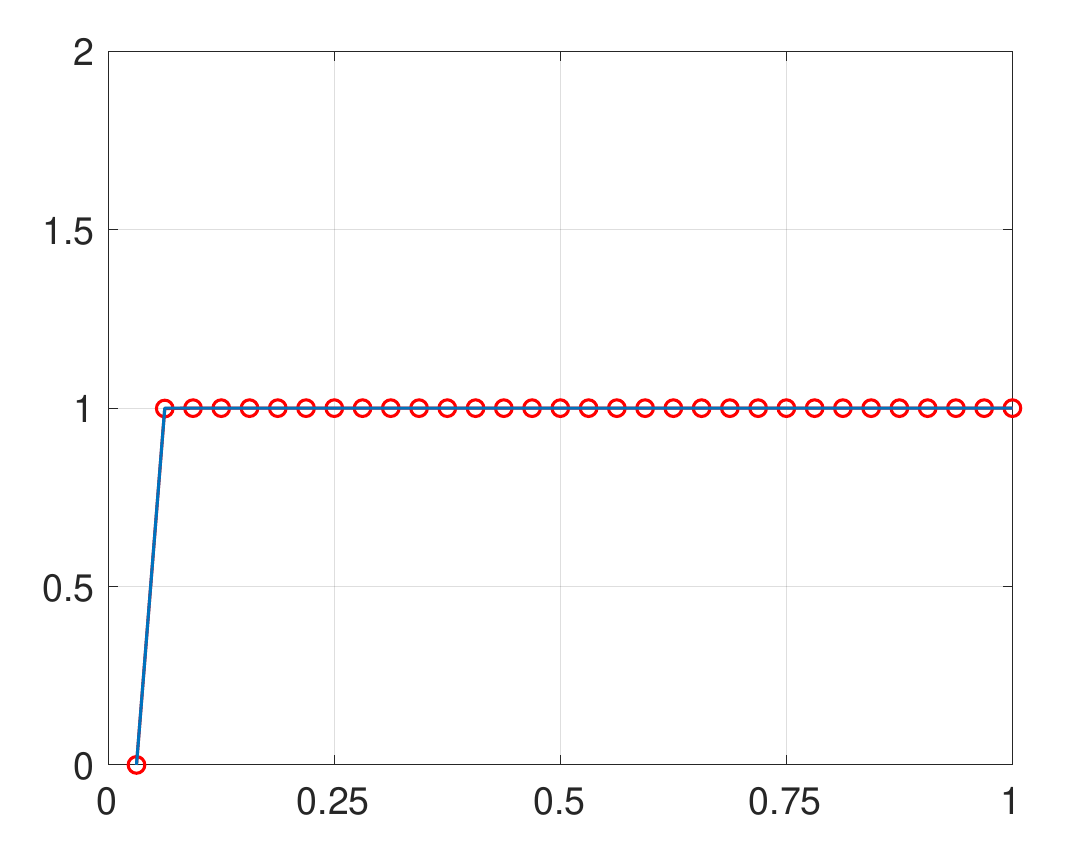}
\\
(a)
&
(b)
\end{tabular}
\caption{Comparison between the analytical solution (in blue) and the numerical solution (in red) for the pressure field obtained with (a) the P1-P1 finite-element method and (b) the stabilized P1-P1 finite-element method. The hydraulic conductivity is fixed at $K=10^{-6}$, and the grid size is $h=1/32$.}\label{figuras_P1_stabilization}
\end{figure}

To determine the appropriate value of $\gamma$ in order to obtain optimal convergence of the proposed iterative method, we use the following corollary of Theorem~\ref{thm:optimal}.
\begin{corollary}\label{cor_1}
Iterative method \eqref{it_pressure2}-\eqref{it_displacement2} with parameters  $\displaystyle L = \frac{3\alpha^2}{2(\lambda+2\mu)}$ and $\displaystyle\gamma= \frac{2}{3}$ converges in two iterations for the stabilized P1-P1 discretization of Terzaghi's problem.
\end{corollary}
\begin{proof}
The proof is based on explicitly calculating the Schur complement, $S_p$, for the considered stabilized P1-P1 discretization of the Terzaghi problem, and then using Theorem~\ref{thm:optimal}. Specifically, it is easily shown that 
$$\displaystyle S_p -C = -D A^{-1}G = \frac{\alpha^2}{\lambda+2\mu} (\frac{3}{2} M - \frac{1}{2} M_l).$$ 
By using this expression, and choosing  $\displaystyle L = \frac{3\alpha^2}{2(\lambda+2\mu)}$ and $\displaystyle \gamma= \frac{2}{3}$, we observe that the splitting corresponding to the iterative method \eqref{it_pressure2}-\eqref{it_displacement2}  given in \eqref{splitting_gamma} matches with the decomposition \eqref{splitting}  in Theorem~\ref{thm:optimal}, and therefore such scheme  converges in only two iterations.
\end{proof}
The previous result provides the optimal value of $\gamma$ for the proposed iterative scheme. To illustrate this graphically, we choose, without loss of generality, the following values of the physical and discretization parameters: the hydraulic conductivity is  $K=10^{-10}$,  $\lambda + 2 \mu=1$, $\alpha=1$, and the mesh size is chosen as $h=1/32$. As the stopping criterion we use that the norm of the residual is less than $10^{-8}$. In Figure~\ref{terzaghilinear}, we show the number of iterations of the proposed scheme for different values of parameter $\gamma$. As expected from the previous theoretical result, only two iterations are needed when $\gamma=2/3$. Other values provide a higher number of iterations, and even exhibit divergence for values of $\gamma$ lower than $0.4$. 
\begin{figure}[!htb]
\begin{center}
\includegraphics*[width = 0.45\textwidth]{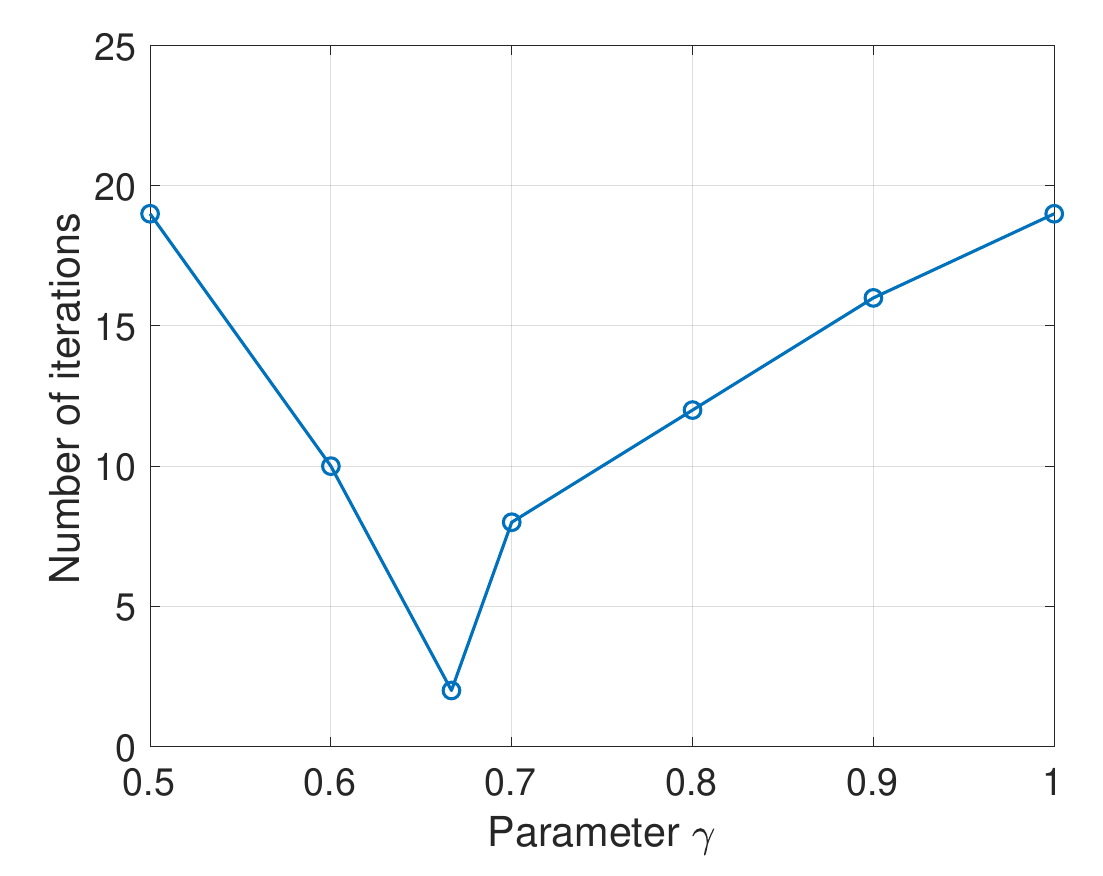}
\caption{Number of iterations of the iterative coupling method for different values of $\gamma$, when the stabilized P1-P1 discretization of Terzaghi's problem is considered. }
\label{terzaghilinear}
\end{center}
\end{figure}

\subsection{Discretization with the MINI-element for displacement}
If the\newline MINI-element is used as the spatial discretization for the displacement (and linears for the pressure), fluid pressure oscillations are once again observed in the numerical solution. Figure~\ref{figuras_mini_stabilization}(a) shows this oscillation at the final time $T_f=0.1$, taking the hydraulic conductivity $K=10^{-6}$, $\lambda + 2\mu=1$, Biot coefficient $\alpha = 1$, and a mesh size $h=1/32$. Adding the proposed stabilization with parameter $\displaystyle L=\frac{\alpha^2}{\lambda+2\mu}$, however, eliminates the oscillations as shown in Figure~\ref{figuras_mini_stabilization}(b). 
\begin{figure}[h!]
\begin{tabular}{cc}
\includegraphics[width = 0.45\textwidth]{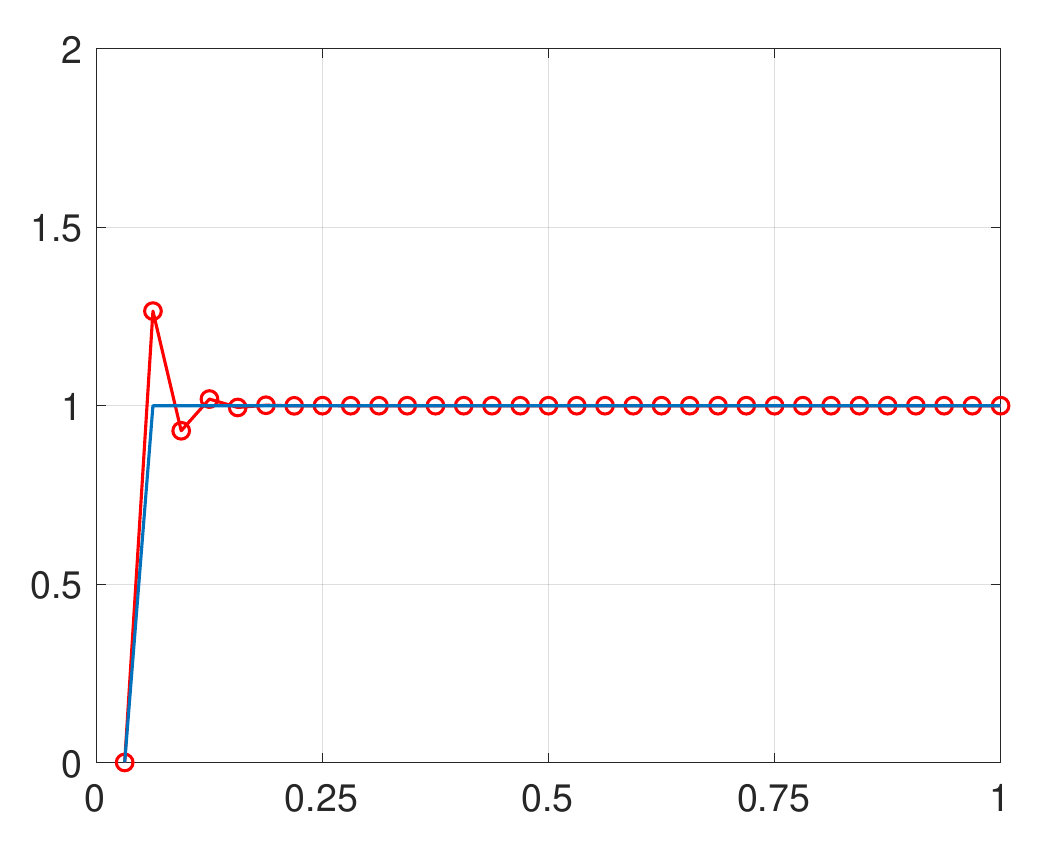}
&
\includegraphics[width = 0.45\textwidth]{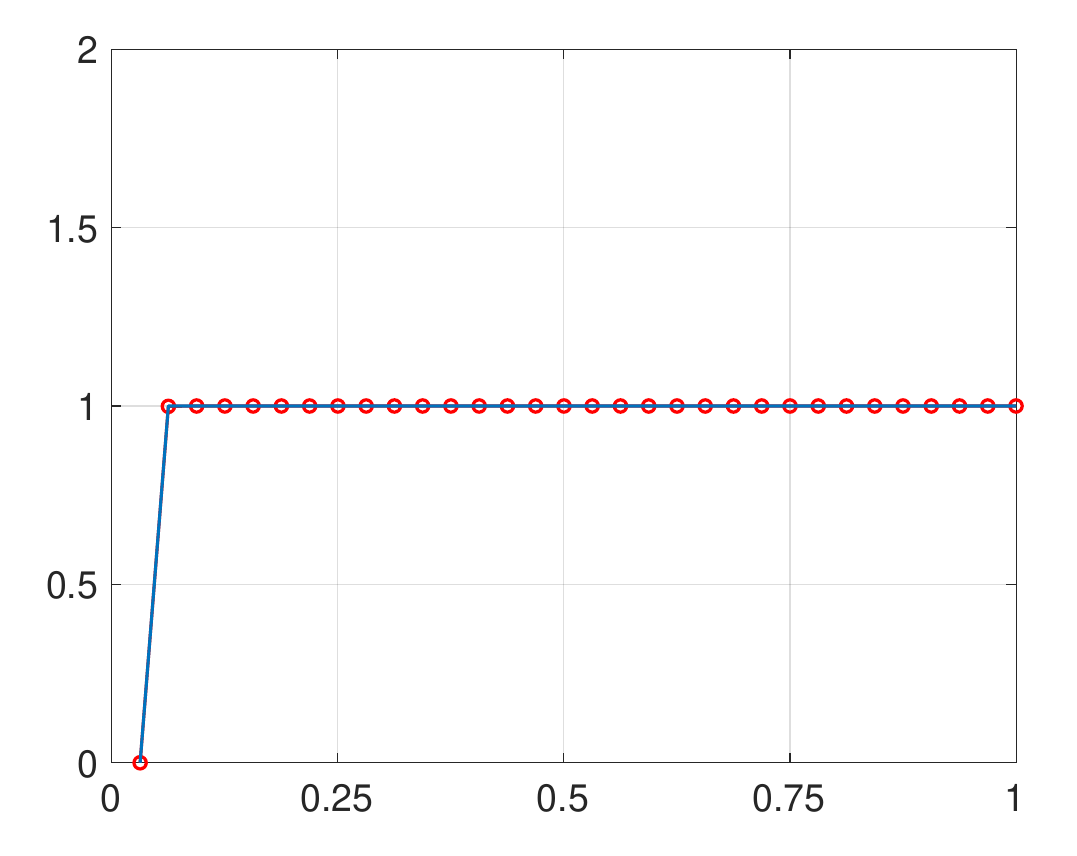}
\\
(a)
&
(b)
\end{tabular}
\caption{Comparison between the analytical solution (in blue) and the numerical solution (in red) for the pressure field obtained with (a) the MINI-element and (b) the stabilized MINI-element for displacement. The hydraulic conductivity is fixed as $K=10^{-6}$, and the grid size is $h=1/32$.}\label{figuras_mini_stabilization}
\end{figure}

Similarly to the linear case, we find the optimal value of $\gamma$ in terms of the convergence of the proposed iterative method, using another corollary of Theorem~\ref{thm:optimal}.
\begin{corollary}\label{cor_2}
Iterative method \eqref{it_pressure2}-\eqref{it_displacement2} with parameters  $\displaystyle L = \frac{\alpha^2}{\lambda+2\mu}$ and $\gamma= 1$ converges in two iterations for the stabilized MINI-element discretization of the Terzaghi's problem. 
\end{corollary}
\begin{proof}
In this case, we once again directly compute the Schur complement, 
$$\displaystyle S_p -C = -D A^{-1}G = \frac{\alpha^2}{\lambda+2\mu} M.$$ 
Choosing  $\displaystyle L = \frac{\alpha^2}{(\lambda+2\mu)}$ and $\gamma= 1$, we observe that the splitting corresponding to the iterative method \eqref{it_pressure2}-\eqref{it_displacement2}  given in \eqref{splitting_gamma} matches with the decomposition \eqref{splitting}  in Theorem~\ref{thm:optimal}, and, therefore, such scheme  converges in only two iterations.
\end{proof}
Note that for the stabilized MINI-element the optimal parameter for the iterative algorithm is $\gamma=1$ and, therefore, our approach reduces to
a simpler one.  Again, in order to illustrate this result graphically, we take the hydraulic conductivity as $K=10^{-10}$, $\lambda + 2 \mu=1$, $\alpha=1$, and the mesh size as $h=1/32$, and plot the number of iterations of the proposed scheme for different values of parameter $\gamma$ in Figure~\ref{terzaghimini}.  As the stopping criterion, we use that the norm of the residual is less than $10^{-8}$.  As expected from the previous theoretical result, only two iterations are needed when $\gamma=1$. Other values provide a higher number of iterations.

\begin{figure}[!htb]
\begin{center}
\includegraphics*[width = 0.45\textwidth]{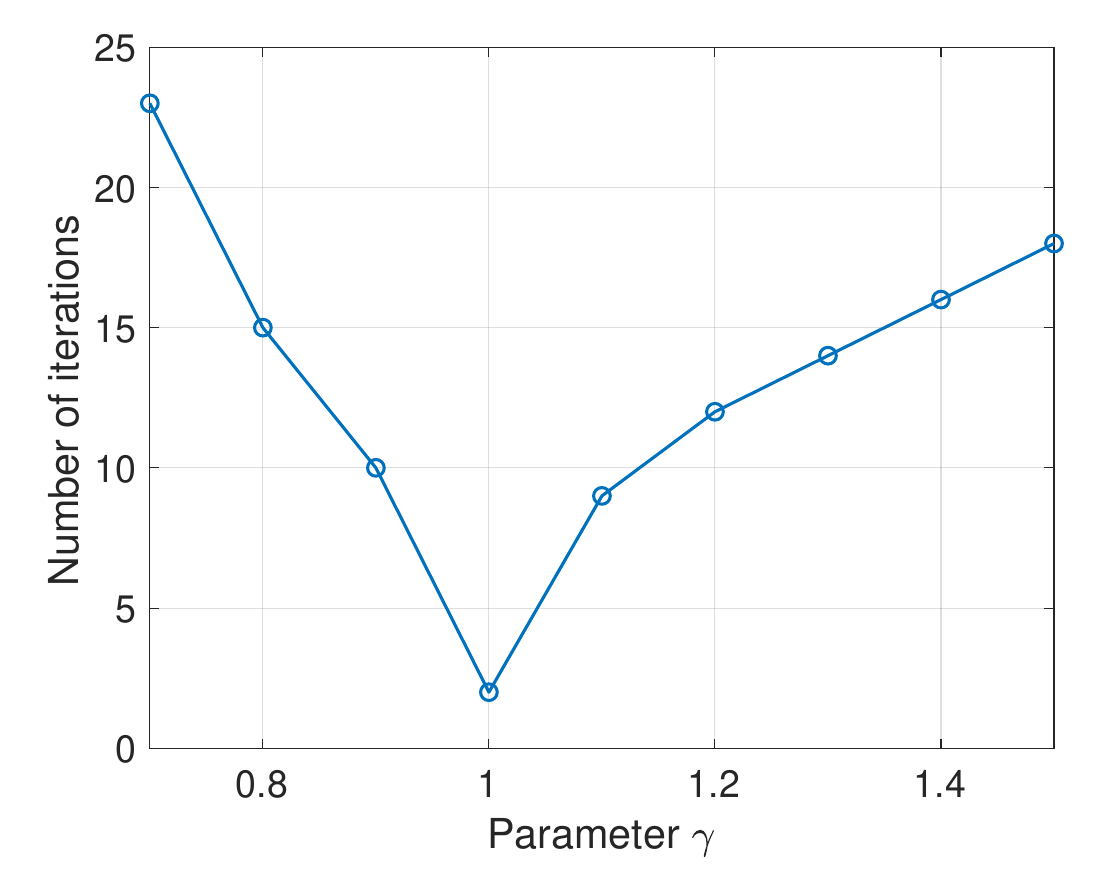}
\caption{Number of iterations of the iterative coupling method for different values of $\gamma$, when the stabilized MINI-element discretization of Terzaghi's problem is considered. }
\label{terzaghimini}
\end{center}
\end{figure}

\section{Numerical experiments}\label{sec:numerical}
\setlength{\belowcaptionskip}{5pt}
\setlength{\abovecaptionskip}{5pt}
In this section, we present numerical results confirming the robustness of our proposed algorithms with respect to physical and discretization parameters for both two- and three-dimensional problems. With this purpose, two well-known benchmark problems from the literature are considered, namely the two-dimensional Barry \& Mercer's problem~\cite{Barry} and a three-dimensional footing problem (see for example~\cite{NLA:NLA587}).
The footing test problems were implemented in the HAZmath library~\cite{hazmath}, which contains routines for finite elements, multilevel solvers, and graph
algorithms, and the numerical tests were performed on a workstation with an 8-core 3 GHz Intel Xeon “Sandy Bridge” CPU and 32 GB of RAM per core.

\subsection{Two-dimensional Barry \& Mercer problem}
A well-known poroelastic benchmark test on a finite two-dimensional domain is Barry \& Mercer's problem~\cite{Barry}. It models the behavior of a rectangular uniform porous material $[0,a]\times[0,b]$ with a pulsating point source, drained on all sides, and with zero tangential displacements assumed on the whole boundary. The point-source corresponds to a sine wave and is given by \(f(t) = 2\upsilon \,\delta_{(x_0,y_0)}\sin(\upsilon \,t)\), where $\upsilon = \displaystyle\frac{(\lambda + 2\mu)K}{a\,b}$ and $\delta_{(x_0,y_0)}$ is the Dirac delta at the point $(x_0,y_0)$. In Figure~\ref{BarryMercer}, the computational domain and the boundary conditions are depicted.
\begin{figure}[!htb]
\begin{center}
\includegraphics*[width = 0.45\textwidth]{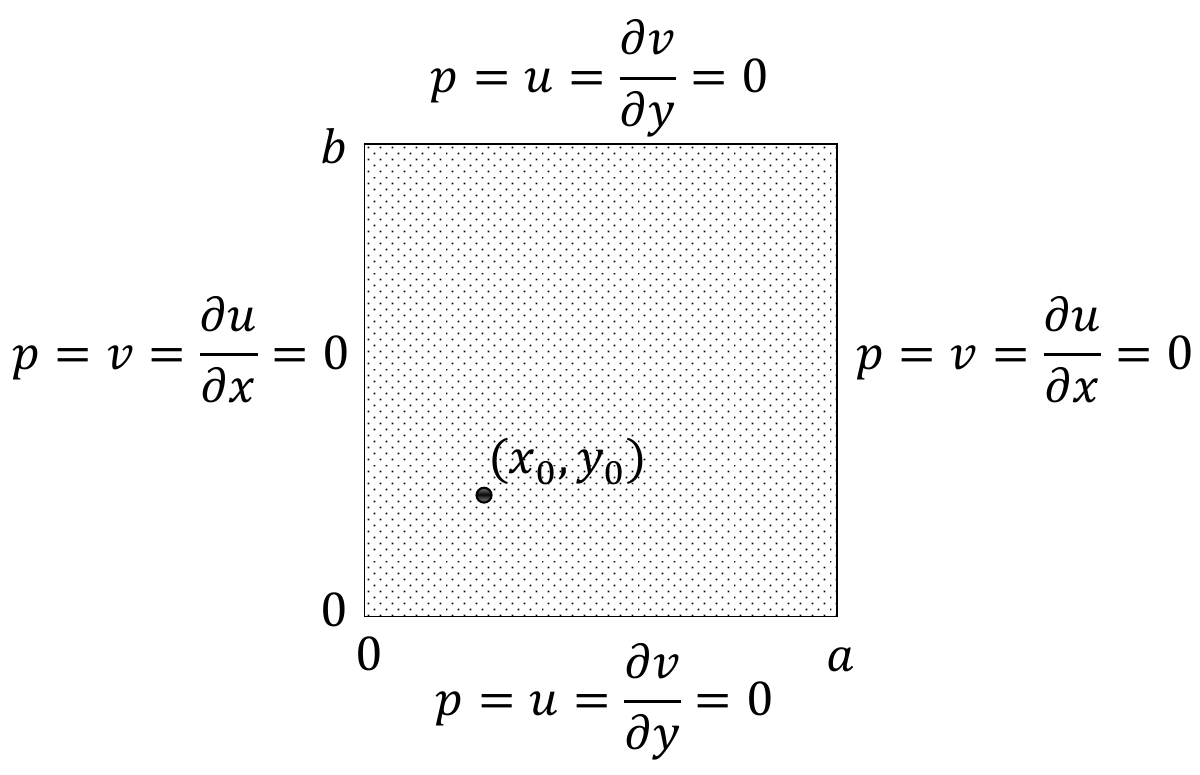}
\caption{Computational domain and boundary conditions for the Barry and Mercer's source problem.}
\label{BarryMercer}
\end{center}
\end{figure}

We consider the square domain $(0,1)\times(0,1)$, and the following values of the material parameters: $E = 10^5$, $\nu = 0.1$, $\alpha=1$, $\beta=10^{8}$ and $K = 10^{-6}$.  The source is positioned at the point $(1/4,1/4)$, and a right triangular grid of mesh size $h=2^{-6}$ is used for the simulations. Fluid pressure oscillations for the Barry and Mercer's problem can be observed by considering standard P1-P1 or MINI element discretizations. In Figure~\ref{BM_oscillations}, we show the numerical solutions obtained for the pressure field at a final time $T_f=10^{-4}$, with only one time step, by using P1-P1 (Fig.~\ref{fig:p1-p1-o}--\ref{fig:p1-p1-no-o}) and the MINI element (Fig.~\ref{fig:MINI-o}--\ref{fig:MINI-no-o}). We observe that without a stabilization term (Fig.~\ref{fig:p1-p1-o} and Fig.~\ref{fig:MINI-o}) non-physical oscillations appear near the source-point.  Adding the proposed stabilization, however, eliminates the oscillations (Fig.~\ref{fig:p1-p1-no-o} and Fig.~\ref{fig:MINI-no-o}).

\begin{figure}[!htb]
     \centering
     \begin{subfigure}[b]{0.3\textwidth}
         \centering
         \includegraphics[width=\textwidth]{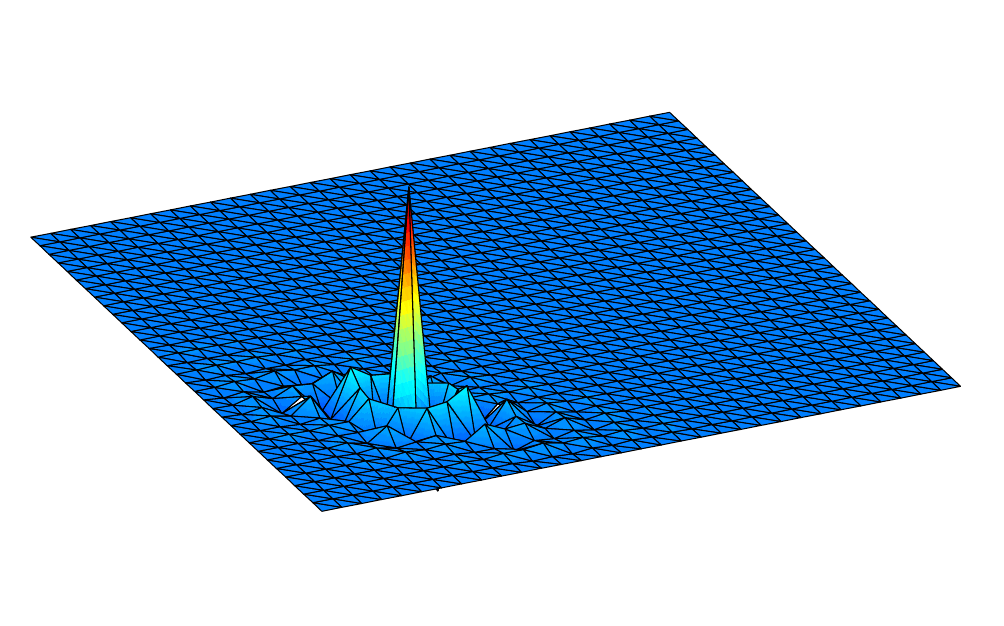}
         \caption{P1-P1 with no stabilization}
         \label{fig:p1-p1-o}
     \end{subfigure}
     \hspace*{1cm}
     \begin{subfigure}[b]{0.3\textwidth}
         \centering
         \includegraphics[width=\textwidth]{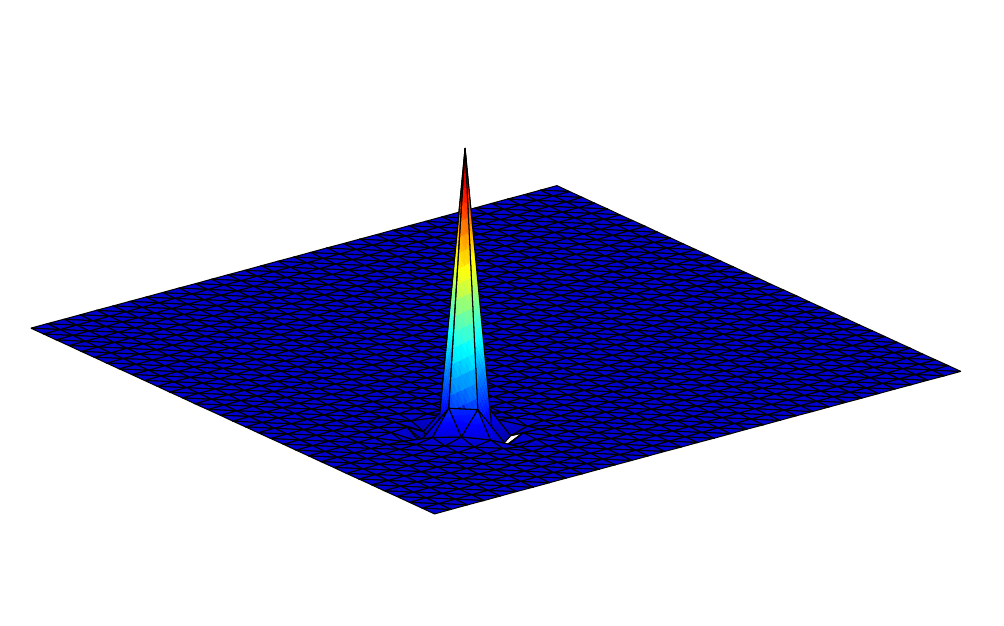}
         \caption{P1-P1 with stabilization}
         \label{fig:p1-p1-no-o}
     \end{subfigure}
          
     \begin{subfigure}[t]{0.3\textwidth}
         \centering
         \includegraphics[width=\textwidth]{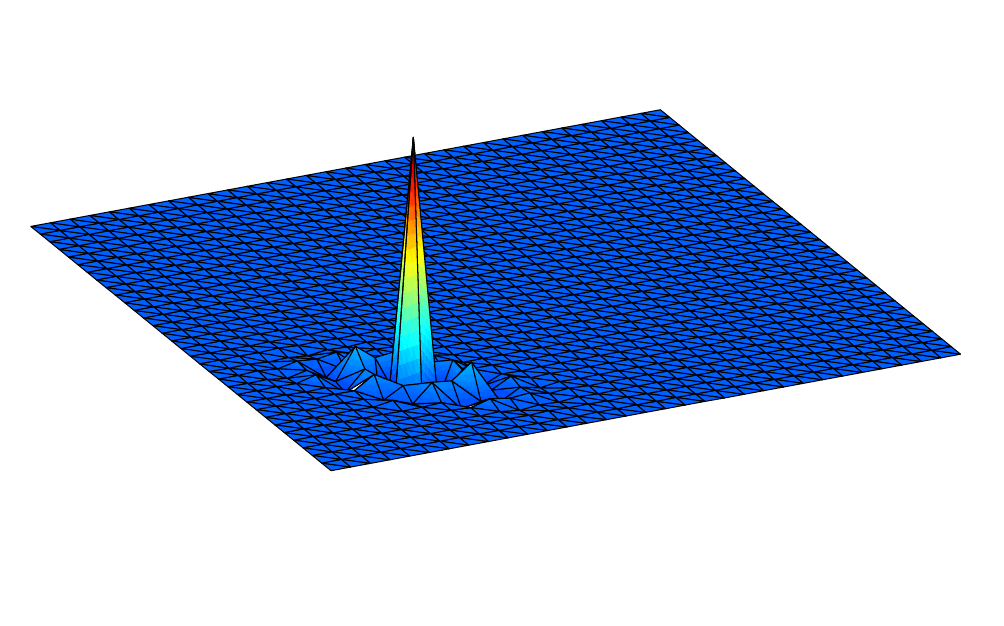}
         \caption{MINI with no stabilization}
         \label{fig:MINI-o}
     \end{subfigure}
          \hspace*{1cm}
     \begin{subfigure}[t]{0.3\textwidth}
         \centering
         \includegraphics[width=\textwidth]{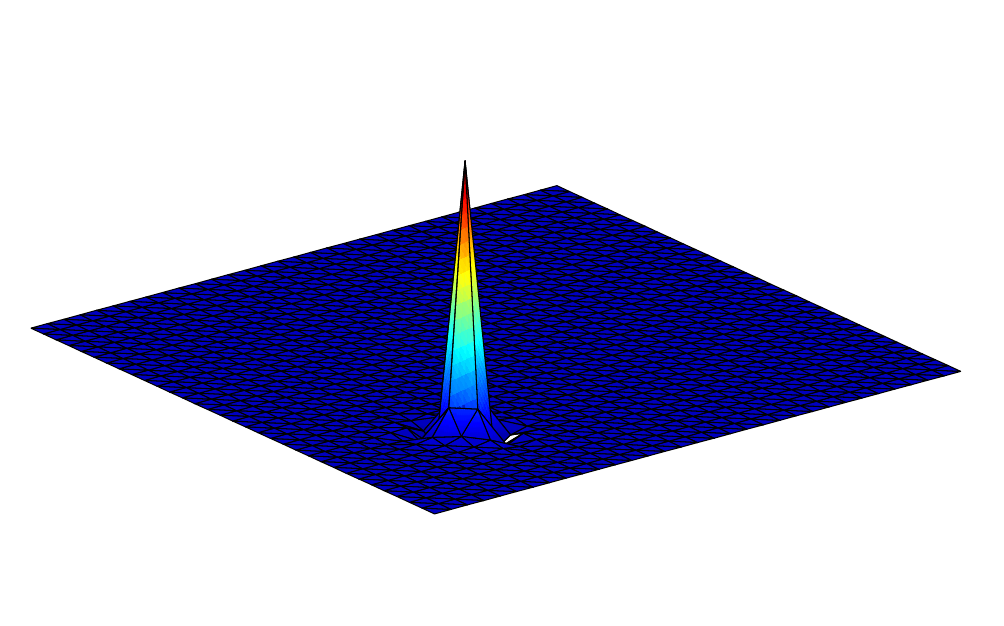}
         \caption{MINI with  stabilization}
         \label{fig:MINI-no-o}
     \end{subfigure}
\caption{Numerical solution (pressure field) at the final time $T_f=10^{-4}$ with hydraulic conductivity $K = 10^{-6}$, Poisson ratio $\nu = 0.1$, and mesh spacing $h=1/64$.}
\label{BM_oscillations}
\end{figure}

When applying the iterative coupling method proposed, we choose the values of parameters $L$ and $\gamma$ which were demonstrated to be optimal in the one-dimensional case. While this does not mean they are optimal for the two-dimensional case here, experimentally we see that they provide almost optimal results in all the tests carried out. 
We first consider the stabilized linear finite-element discretization. In order to illustrate the suitability of the choice of $L$ and $\gamma$, we consider two different test cases with different values of the Poisson ratio $\nu$ and hydraulic conductivity $K$ and different grid sizes. In Figure \ref{Parameters_P1}, we fix $h = 1/32$ and show the number of iterations of the proposed scheme for different values of parameter $\gamma$ and two values of hydraulic conductivity: $K= 10^{-10}$ and $K=10^{-2}$. Figure \ref{Parameters_P1}(a) shows results for $\nu=0.2$ and Figure \ref{Parameters_P1}(b) shows results for $\nu=0.4$.
From both figures, we observe that the choice of $\gamma$ is not important when $K$ is large, whereas for small values of $K$ it is crucial to choose a proper value of parameter $\gamma$. Also, in contrast to the one-dimensional case, here $\gamma = 2/3$ is not optimal for both values of $\nu$ (it is for $\nu=0.4$ but not for $\nu=0.2$). However, the choice of  $\gamma = 2/3$ provides very good, almost optimal, results as is seen in both pictures. 
\begin{figure}[h!]     
\begin{tabular}{cc}
\includegraphics[width = 0.45\textwidth]{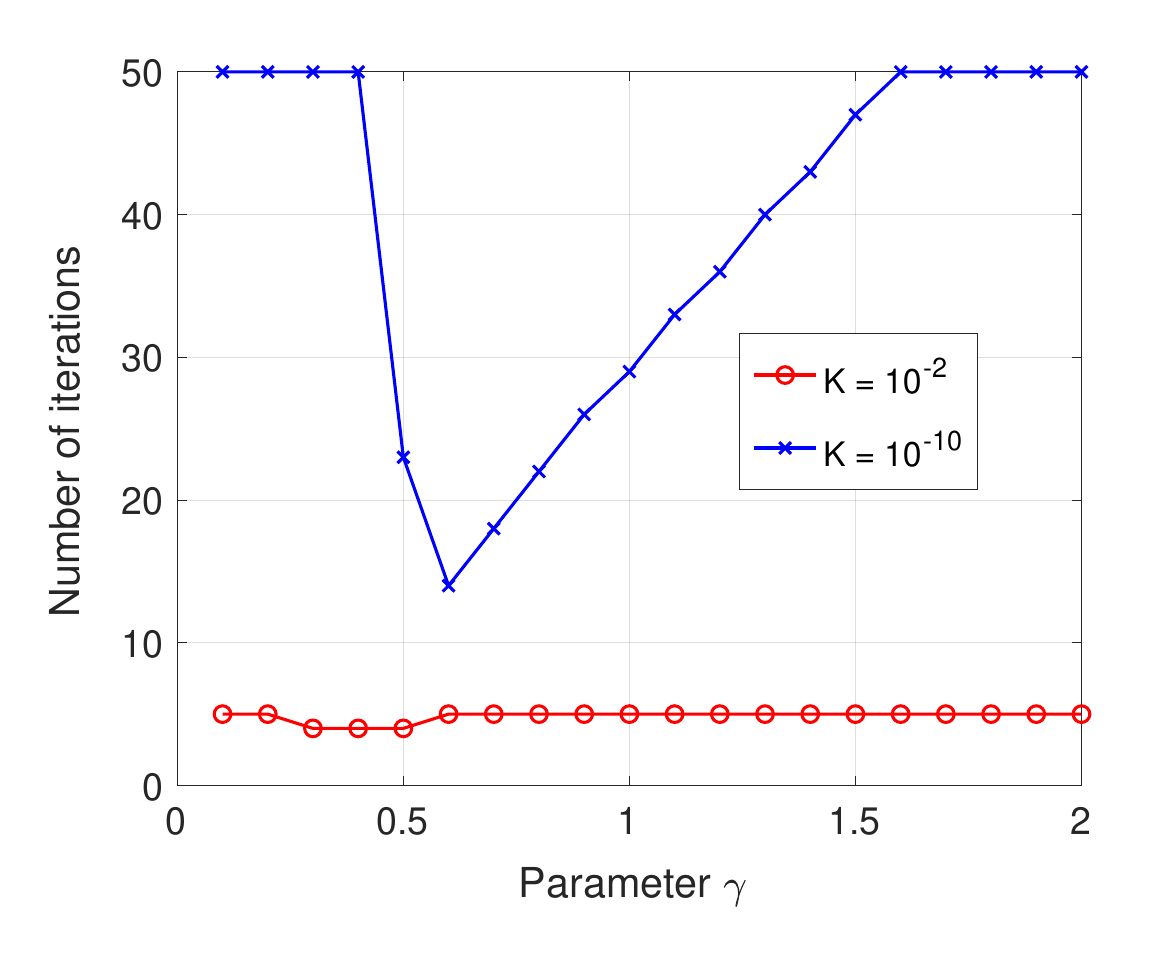}
&
\includegraphics[width = 0.45\textwidth]{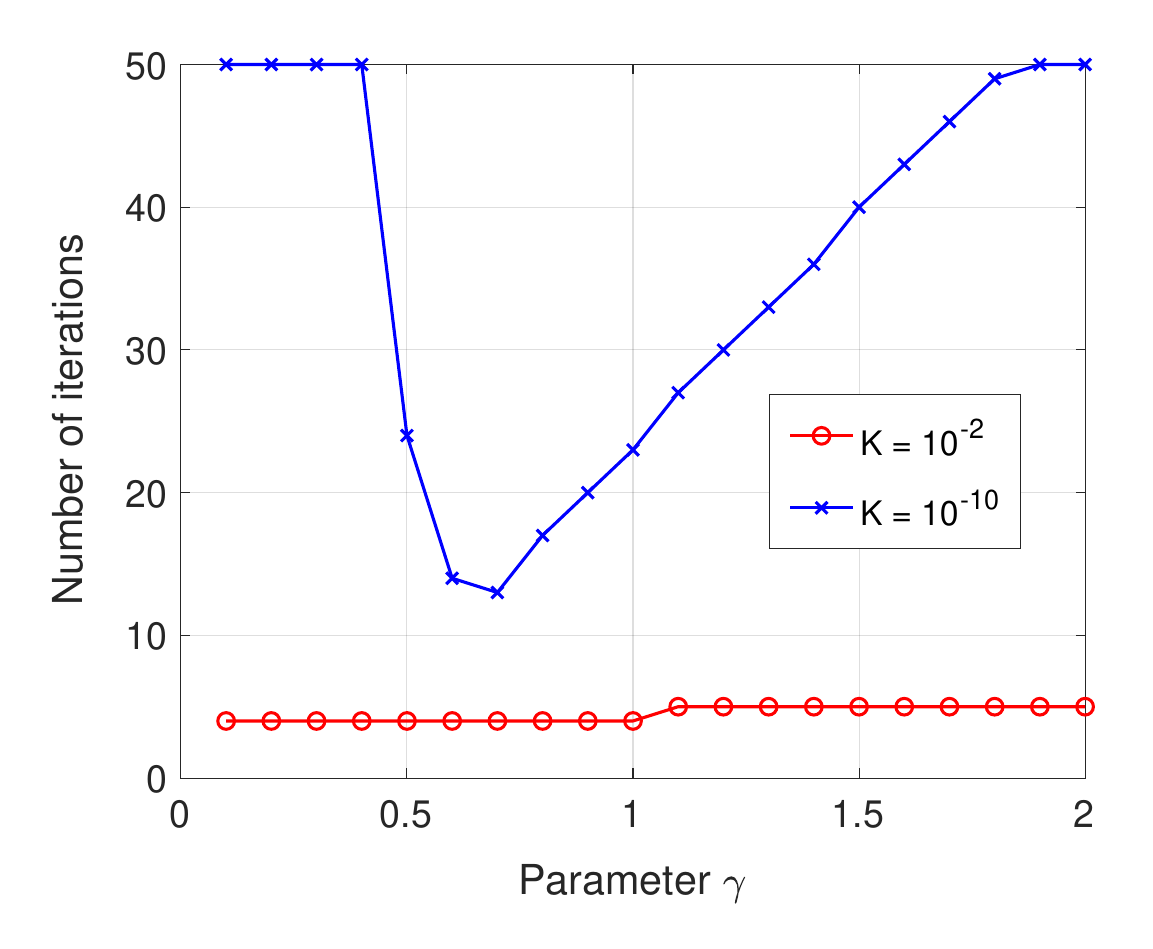}
\\
(a)
&
(b)
\end{tabular}
\caption{Number of iterations of the iterative coupling scheme for different values of $\gamma$, when the stabilized P1-P1 discretization is considered on the Barry \& Mercer's problem. Two different test cases are shown: (a) $\nu=0.2$ and (b) $\nu=0.4$.}\label{Parameters_P1}
\end{figure}

Thus with this choice of parameter $\gamma$, in Table \ref{Tabla_BM_P1}(a), we show the number of iterations of Algorithm \eqref{it_pressure2}-\eqref{it_displacement2} using the stabilized P1-P1 discretization for several values of the hydraulic conductivity $K$ and the mesh size $h$. We use $\| \delta p^{n,i} \| + \| \delta {\bf u}^{n,i} \| \leq 10^{-8}$ as the stopping criterion, where $\delta p_h^{n,i} = p_h^{n,i} - p_h^{n,i-1}$ and $\delta {\bf u}_h^{n,i} = {\bf u}_h^{n,i} - {\bf u}_h^{n,i-1}$ denote the difference between two successive approximations for displacements and for pressure, respectively.  The number of iterations stays relatively stable for all the different values of $K$ and $h$, demonstrating the robustness of the proposed iterative coupling method, i.e., Algorithm \eqref{it_pressure2}-\eqref{it_displacement2}.
In Table~\ref{Tabla_BM_P1}(b), we show the robustness of the proposed iterative solver with respect to the Poisson ratio. For different values of $\nu$ varying from $0.1$ to $0.49$ and different mesh sizes, we display the number of iterations of Algorithm \eqref{it_pressure2}-\eqref{it_displacement2} in order to fulfill the same stopping criterion as above. Again, the observed results are robust with respect to the considered parameters. 


\begin{table}[!htbp]
\begin{tabular}{cc}
\begin{minipage}{0.46\textwidth}
\begin{center}
		\begin{tabular}{|c|c|c|c|c|}  
			\cline{1-5}
		         $K$/$h$ & $1/16$ & $1/32$ & $1/64$ & $1/128$\\
			\cline{1-5}
			\hline
			\multicolumn{1}{|c|}{$10^{-2}$}   & 4 & 4 & 4 & 4 \\ 		
			\multicolumn{1}{|c|}{$10^{-4}$}   & 6 &  6 & 6  &  6  \\ 
			\multicolumn{1}{|c|}{$10^{-6}$}   & 11 & 11  & 11 & 11 \\ 
			\multicolumn{1}{|c|}{$10^{-8}$}  & 15 & 15   &  15 & 15   \\ 
			\multicolumn{1}{|c|}{$10^{-10}$}  & 11   & 11   &  12 &  12  \\ 
			\multicolumn{1}{|c|}{$10^{-12}$}  & 6 &  7  &  7 &  8 \\ 
			\hline
		\end{tabular} \\ [1ex]
		(a)
\end{center}
\end{minipage}
&
\begin{minipage}{0.46\textwidth}
\begin{center}
		\begin{tabular}{|c|c|c|c|c|}  
			\cline{1-5}
		         $\nu$/$h$ & $1/16$ & $1/32$ & $1/64$ & $1/128$\\
			\cline{1-5}
			\hline
			\multicolumn{1}{|c|}{$0.1$}   & 18 & 20 &  21 & 22 \\ 		
			\multicolumn{1}{|c|}{$0.2$}   &  16& 17 & 18 &  19  \\ 
			\multicolumn{1}{|c|}{$0.3$}   &  13 &  14& 15 & 16 \\ 
			\multicolumn{1}{|c|}{$0.4$}  &  11  & 11 & 12 & 12\\ 
			\multicolumn{1}{|c|}{$0.49$}  & 12  & 11 & 9 & 8 \\ 
			\hline
		\end{tabular} \\ [1ex]
		(b)
\end{center}
\end{minipage}
\end{tabular}		
\caption{Number of iterations of Algorithm \eqref{it_pressure2}-\eqref{it_displacement2} with $\gamma=2/3$, $L=3\alpha^2/(2(\lambda+\mu))$, and $E=10^5$, using the stabilized P1-P1 discretization in two dimensions, for different mesh sizes $h$ and (a) different values of the hydraulic conductivity $K$ (with $\nu=0.4$) and (b) different values of Poisson ratio $\nu$ (with $K=10^{-10}$).}
\label{Tabla_BM_P1}
\end{table}

Next, we consider the stabilized MINI-element scheme. In this case, the one-dimensional study provided us the value of $\gamma=1$ as the optimal one for the convergence of the iterative method. Again, we consider two different values of the Poisson ratio: $\nu=0.2$ and $\nu=0.4$, and two different values of the hydraulic conductivity: $K=10^{-2}$ and  $K=10^{-10}$.  Figure \ref{Parameters_MINI}(a) displays the number of iterations of the iterative coupling method with $\nu=0.2$ and $h = 1/32$ for different values of parameter $\gamma$. Figure \ref{Parameters_MINI}(b) displays similar results for $\nu =0.4$ and grid size $h=1/64$.  We observe that although $\gamma = 1$ is not the optimal value as it was in the one-dimensional case, it does provide reasonably good results. 

\begin{figure}[h!]
\begin{tabular}{cc}
\includegraphics[width = 0.45\textwidth]{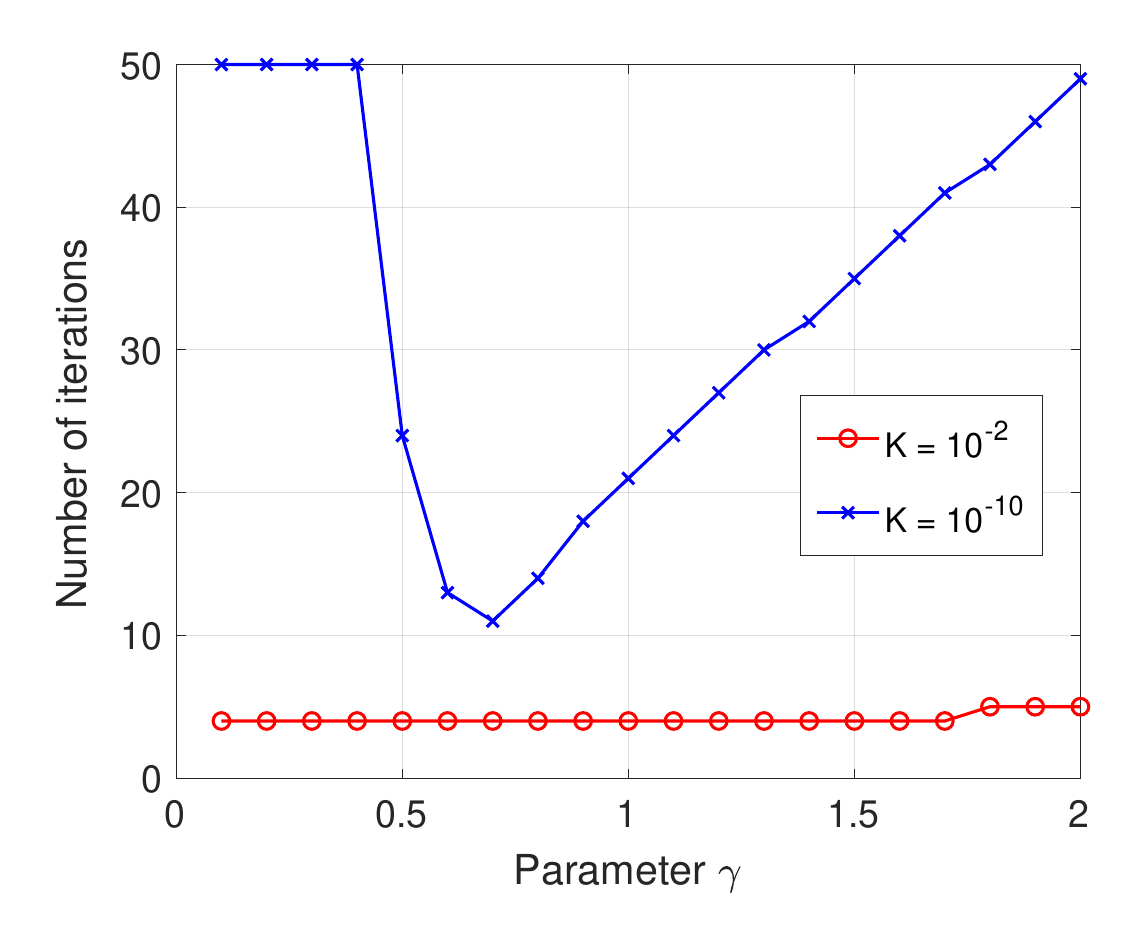}
&
\includegraphics[width = 0.45\textwidth]{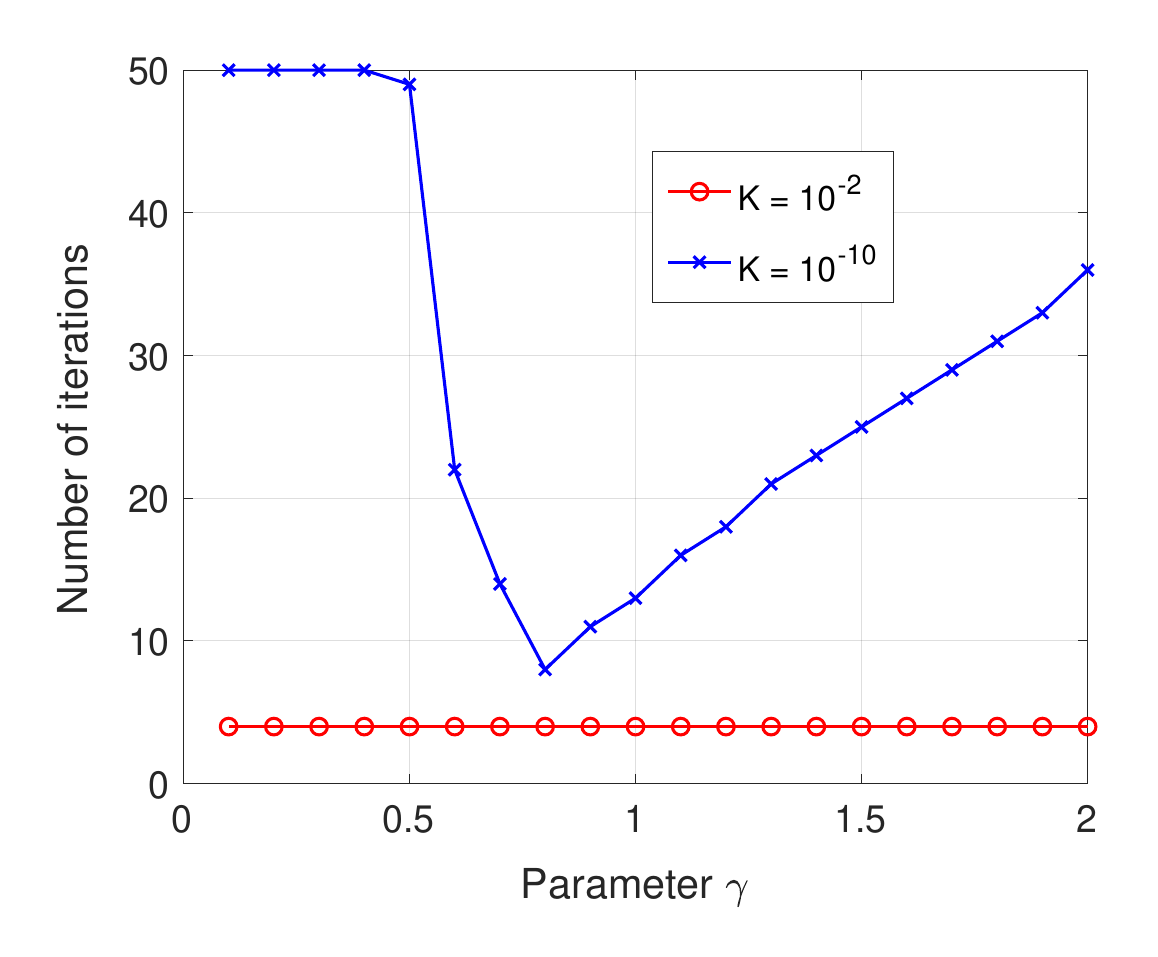}
\\
(a)
&
(b)
\end{tabular}
\caption{Number of iterations of the iterative coupling scheme for different values of $\gamma$, when the stabilized MINI-element discretization is considered on the Barry \& Mercer's problem. Two different test cases are shown: (a) $\nu=0.2$ and (b) $\nu=0.4$}\label{Parameters_MINI}
\end{figure}

Analogously to what was done for linear elements in Table \ref{Tabla_BM_P1}, we perform the same experiments for the stabilized MINI-element, shown in Table \ref{Tabla_BM_MINI}.  The number of iterations of Algorithm \eqref{it_pressure2}-\eqref{it_displacement2}, with the choice of $\gamma=1$, for several values of the hydraulic conductivity $K$ and the mesh size $h$ and for values of $\nu$ are shown. Again, we observe that the number of iterations stays relatively stable for all the different values of the considered parameters, demonstrating the robustness of the proposed iterative coupling method.
\begin{table}[htbp]
\begin{tabular}{cc}
\begin{minipage}{0.46\textwidth}
\begin{center}
		\begin{tabular}{|c|c|c|c|c|}  
			\cline{1-5}
		         $K$/$h$ & $1/16$ & $1/32$ & $1/64$ & $1/128$\\
			\cline{1-5}
			\hline
			\multicolumn{1}{|c|}{$10^{-2}$}   & 4 & 4 & 4 &  4 \\ 		
			\multicolumn{1}{|c|}{$10^{-4}$}   & 5 &  5 &  5 & 5   \\ 
			\multicolumn{1}{|c|}{$10^{-6}$}   & 10 &  10 & 10 &  10 \\ 
			\multicolumn{1}{|c|}{$10^{-8}$}  & 13 & 13  & 13 &   13  \\ 
			\multicolumn{1}{|c|}{$10^{-10}$}  & 14   & 13   & 13  &  13  \\ 
			\multicolumn{1}{|c|}{$10^{-12}$}  & 14 &  14 & 14 &  14 \\ 
			\hline
		\end{tabular}\\ [1ex]
		(a)
\end{center}
\end{minipage}
&
\begin{minipage}{0.46\textwidth}
\begin{center} 
		\begin{tabular}{|c|c|c|c|c|}  
			\cline{1-5}
		         $\nu$/$h$ & $1/16$ & $1/32$ & $1/64$ & $1/128$\\
			\cline{1-5}
			\hline
			\multicolumn{1}{|c|}{$0.1$}   & 25 & 25 &  25 &  24\\ 		
			\multicolumn{1}{|c|}{$0.2$}   & 21 & 21 & 21 & 21   \\ 
			\multicolumn{1}{|c|}{$0.3$}   & 17  & 17 & 17 & 17 \\ 
			\multicolumn{1}{|c|}{$0.4$}  &  14  & 13 & 13 & 13\\ 
			\multicolumn{1}{|c|}{$0.49$}  &  11 & 10 &  9 & 7 \\ 
			\hline
		\end{tabular} \\ [1ex]
		(b)
\end{center}
\end{minipage}
\end{tabular}
\caption{Number of iterations of Algorithm \eqref{it_pressure2}-\eqref{it_displacement2} with $\gamma=1$, $L=\alpha^2/(\lambda+\mu)$, and $E=10^5$, using the stabilized MINI-element scheme in two dimensions, for different mesh sizes $h$ and (a) different values of the hydraulic conductivity $K$ (with $\nu=0.4$) and (b) different values of Poisson ratio $\nu$ (with $K=10^{-10}$).}
\label{Tabla_BM_MINI}
\end{table}

\subsection{Three-dimensional footing problem}
To confirm the results in three-dimensions, we consider a footing problem, as seen in~\cite{NLA:NLA587}. The domain, illustrated in Figure \ref{fig:footing_domain}, is a unit
cube modeling a block of porous soil. A downward uniform load, $\boldsymbol{\sigma}_0$, of intensity $10^4$ per
unit area is applied in a square of size $0.5 \times 0.5$ in the middle of the top face. The
base of the domain is assumed to be fixed, while the rest of the domain is free to
drain. The material properties used are $\alpha = 1$, $\beta = 10^6$ and $E=10^4$.
In the following, we vary the hydraulic conductivity~$K$, and Poisson ratio~$\nu$.  
\begin{figure*}[ht!]
    \centering
        \includegraphics[scale=0.8]{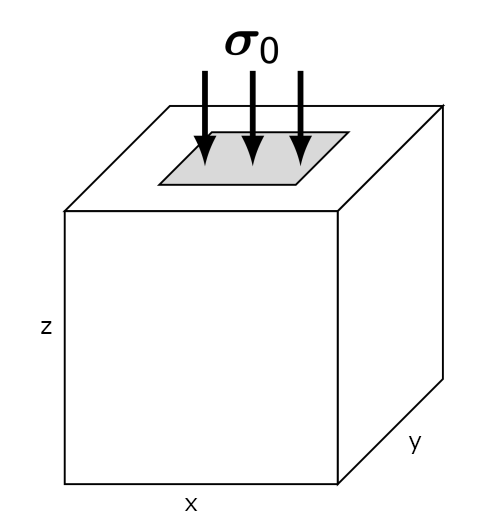}
        \caption{Footing problem domain.}\label{fig:footing_domain}
\end{figure*}


First, similar to the two-dimensional case, we perform several tests to show that the values of parameters $L$ and $\gamma$, which are optimal in the one-dimensional case, still provide almost optimal results in three dimensions.  For the P1-P1 discretization,  fixing $\nu = 0.2$ and $h=1/8$ and varying $K = 10^{-2}$ or $10^{-8}$, we show the number of iterations for different values of parameter $\gamma$ in Figure \ref{Parameters_P1_3D}(a).  In Figure \ref{Parameters_P1_3D}(b), we repeat the same experiments with $\nu = 0.4$.  We observe the same results as in two dimensions, namely, the choice of $\gamma = 2/3$ provides a nearly optimal result as expected. 

\begin{figure}[h!]
	\begin{tabular}{cc}
		\includegraphics[width = 0.45\textwidth]{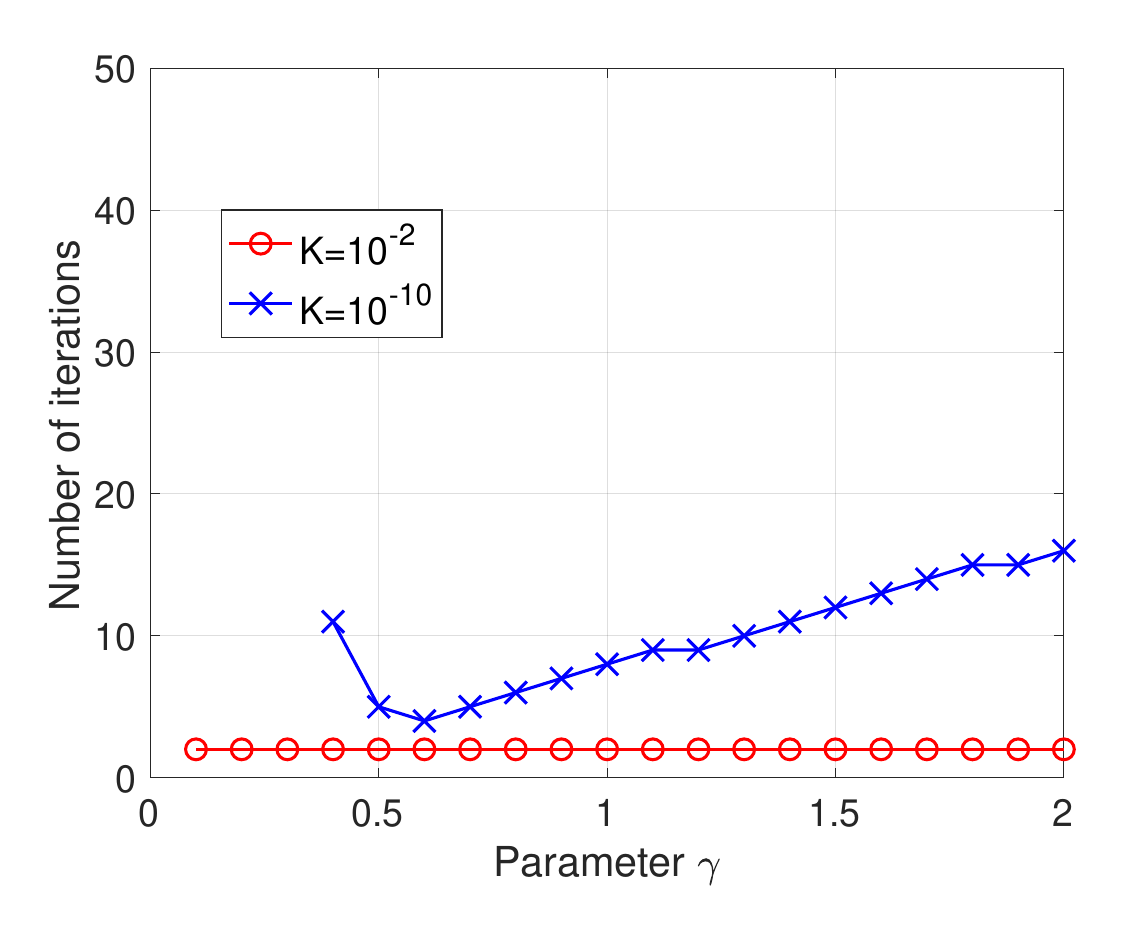}
		&
		\includegraphics[width = 0.45\textwidth]{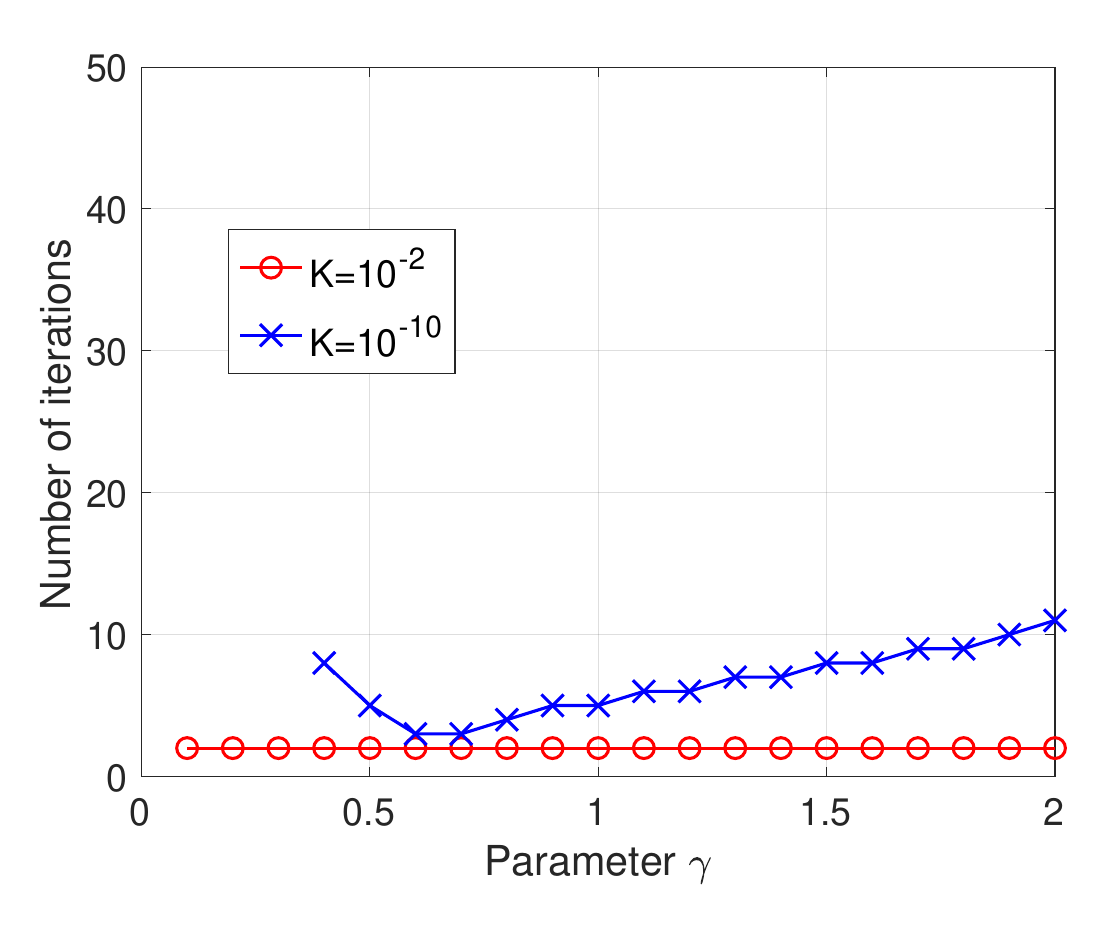}
		\\
		(a)
		&
		(b)
	\end{tabular}
	\caption{Number of iterations of the iterative coupling scheme for different values of $\gamma$, when the stabilized P1-P1 discretization is considered in three dimensions. Two different test cases are shown: (a) $\nu=0.2$ and (b) $\nu=0.4$.}\label{Parameters_P1_3D}
\end{figure}

We repeat the same experiments for the MINI element.  Again, we fix $h = 1/8$ and choose $\nu = 0.2$ or $0.4$ and $K = 10^{-2}$ or $10^{-10}$.  The number of iterations are reported in Figure \ref{Parameters_MINI_3D}.  The observation is basically the same as before and $\gamma=1$ is a nearly optimal choice in three dimensions as well.  

\begin{figure}[h!]
	\begin{tabular}{cc}
		\includegraphics[width = 0.45\textwidth]{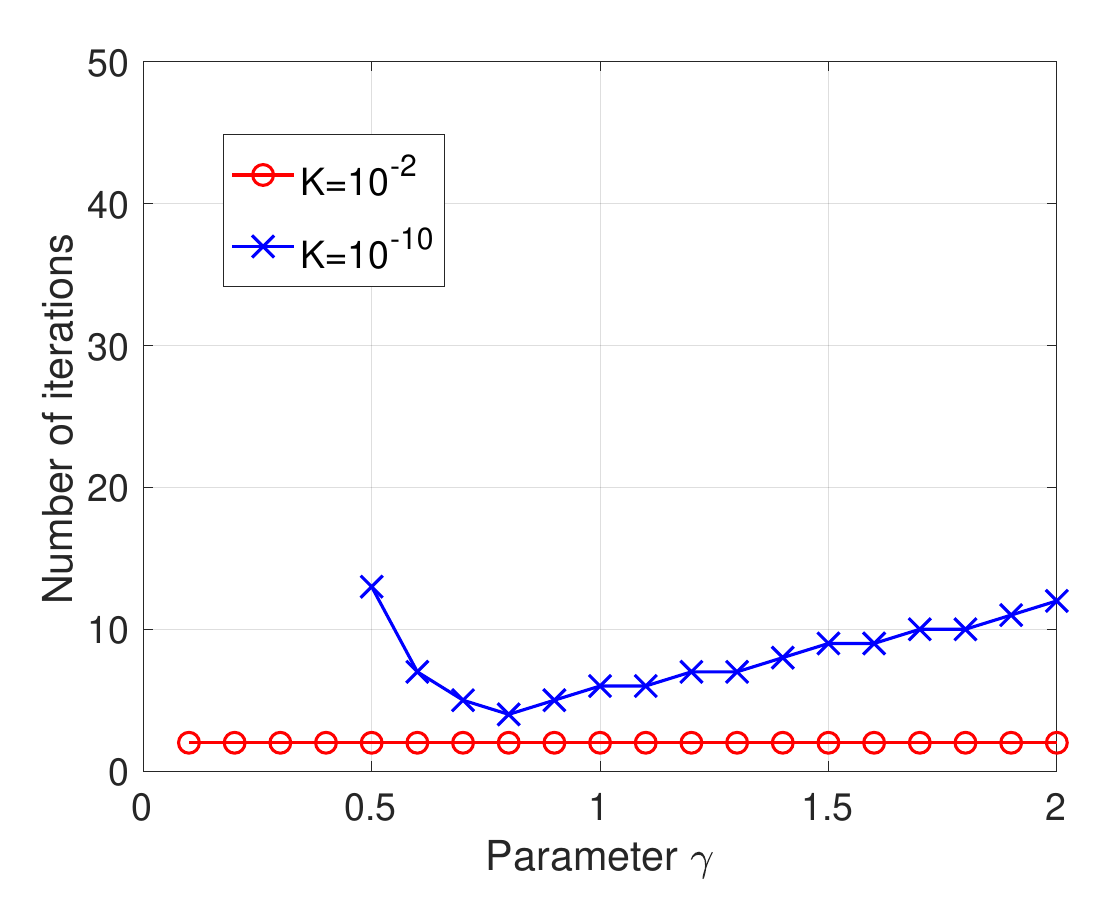}
		&
		\includegraphics[width = 0.45\textwidth]{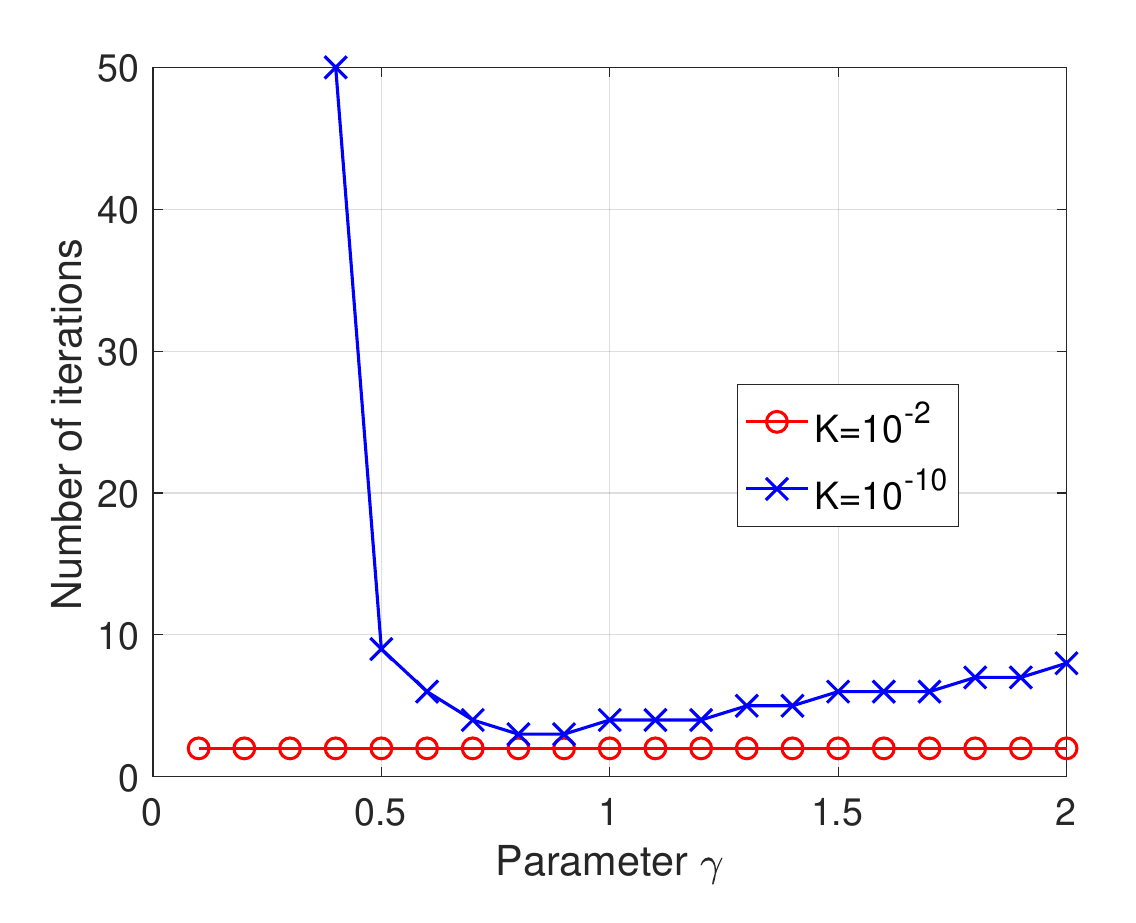}
		\\
		(a)
		&
		(b)
	\end{tabular}
	\caption{Number of iterations of the iterative coupling scheme for different values of $\gamma$, when the stabilized MINI-element discretization is considered in three dimensions. Two different test cases are shown: (a) $\nu=0.2$ and (b) $\nu=0.4$}\label{Parameters_MINI_3D}
\end{figure}

Next, we fix $\gamma = 2/3$ for the P1-P1 discretization and $\gamma = 1$ for the MINI element and test the robustness of the proposed iterative scheme with respect to $h$, $K$, and $\nu$.  From the results in Table \ref{Tabla_BM_P1_3D}, we see that the number of iterations is quite stable with respect to $h$, $K$, and $\nu$ for the P1-P1 element.  In addition, results in Table \ref{Tabla_BM_MINI_3D} demonstrate the robustness of the iterative method for the MINI-element.  Overall, we see that our proposed iterative scheme is robust with respect to all the parameters, which demonstrates its effectiveness in practical applications. 

\begin{table}[htbp]
\begin{tabular}{cc}
\begin{minipage}{0.45\textwidth}
\begin{center}
		\begin{tabular}{|c|c|c|c|}  
			\cline{1-4}
			$K$/$h$ & $1/4$ & $1/8$ & $1/16$ \\
			\cline{1-4}
			\hline
			\multicolumn{1}{|c|}{$10^{-2}$}   & 2 & 2 & 2 \\ 		
			\multicolumn{1}{|c|}{$10^{-4}$}   & 2 &  2 & 2    \\ 
			\multicolumn{1}{|c|}{$10^{-6}$}   & 3 & 3  & 3  \\ 
			\multicolumn{1}{|c|}{$10^{-8}$}  & 3 & 3  &  3    \\ 
			\multicolumn{1}{|c|}{$10^{-10}$}  & 3   & 3   &  3   \\ 
			\multicolumn{1}{|c|}{$10^{-12}$}  & 3 &  3  &  3 \\ 
			\hline
		\end{tabular} \\ [1ex]
		(a)
\end{center}
\end{minipage}
&
\begin{minipage}{0.45\textwidth}
\begin{center}
		\begin{tabular}{|c|c|c|c|}  
			\cline{1-4}
			$\nu$/$h$ & $1/4$ & $1/8$ & $1/16$ \\
			\cline{1-4}
			\hline
			\multicolumn{1}{|c|}{$0.1$}   & 5 & 6 &  6  \\ 		
			\multicolumn{1}{|c|}{$0.2$}   &  4& 5 & 5   \\ 
			\multicolumn{1}{|c|}{$0.3$}   &  4 &  4& 4  \\ 
			\multicolumn{1}{|c|}{$0.4$}  &  3  & 3 & 3 \\ 
			\multicolumn{1}{|c|}{$0.49$}  & 2  & 2 & 2  \\ 
			\hline
		\end{tabular} \\ [1ex]
		(b)
\end{center}
\end{minipage}
\end{tabular}
\caption{Number of iterations of Algorithm \eqref{it_pressure2}-\eqref{it_displacement2} with $\gamma=2/3$, $L=3\alpha^2/(2(\lambda+(2\mu/3)))$, and $E=10^5$, using the stabilized P1-P1 discretization in three dimensions, for different values of mesh size $h$ and (a) different values of the hydraulic conductivity $K$ (with $\nu=0.4$) and (b) different values of Poisson ratio $\nu$ (with $K=10^{-10}$).}
\label{Tabla_BM_P1_3D}
\end{table}

\begin{table}[htbp]
\begin{tabular}{cc}
\begin{minipage}{0.45\textwidth}
\begin{center}
		\begin{tabular}{|c|c|c|c|}  
			\cline{1-4}
			$K$/$h$ & $1/4$ & $1/8$ & $1/16$\\
			\cline{1-4}
			\hline
			\multicolumn{1}{|c|}{$10^{-2}$}   & 2 & 2 & 2  \\ 		
			\multicolumn{1}{|c|}{$10^{-4}$}   & 2 &  2 &  2   \\ 
			\multicolumn{1}{|c|}{$10^{-6}$}   & 3 &  3 & 3  \\ 
			\multicolumn{1}{|c|}{$10^{-8}$}  & 4 & 4  & 3   \\ 
			\multicolumn{1}{|c|}{$10^{-10}$}  & 4   & 4   & 3   \\ 
			\multicolumn{1}{|c|}{$10^{-12}$}  & 4 &  4 & 3  \\ 
			\hline
		\end{tabular} \\ [1ex]
		(a)
\end{center}
\end{minipage}
&
\begin{minipage}{0.45\textwidth}
\begin{center}
		\begin{tabular}{|c|c|c|c|}  
			\cline{1-4}
			$\nu$/$h$ & $1/4$ & $1/8$ & $1/16$ \\
			\cline{1-4}
			\hline
			\multicolumn{1}{|c|}{$0.1$}   & 6 & 6&  7 \\ 		
			\multicolumn{1}{|c|}{$0.2$}   & 5 & 6 & 5    \\ 
			\multicolumn{1}{|c|}{$0.3$}   & 4  & 5 & 4  \\ 
			\multicolumn{1}{|c|}{$0.4$}  &  4  & 4 & 3 \\ 
			\multicolumn{1}{|c|}{$0.49$}  &  2 & 2 &  2  \\ 
			\hline
		\end{tabular} \\ [1ex]
		(b)
\end{center}
\end{minipage}
\end{tabular}
\caption{Number of iterations of Algorithm \eqref{it_pressure2}-\eqref{it_displacement2} with $\gamma=1$, $L=\alpha^2/(\lambda+(2\mu/3))$, and $E=10^5$, using the stabilized MINI-element scheme in three dimensions, for different mesh sizes $h$ and (a) different values of the hydraulic conductivity $K$ (with $\nu=0.4$) and (b) different values of Poisson ratio $\nu$ (with $K=10^{-10}$).}
\label{Tabla_BM_MINI_3D}
\end{table}

\section{More general fluid regimes}\label{sec:two_parameters} 

In this section, we generalize the proposed strategy for a more general case which includes applications where $\beta^{-1}$ is not close to zero. Here, the term $\beta^{-1}$ affects the choice of the stabilization coefficient, $L$. In particular, the optimal values of $L$ to remove the oscillations in the one-dimensional Terzaghi's problem are given by
$$L = \displaystyle \left\{ \begin{array}{ll} \displaystyle \frac{3\alpha^2}{2(\lambda+2\mu)} + \frac{1}{\beta}, & \hbox{for P1-P1,} \\  \displaystyle \frac{\alpha^2}{\lambda+2\mu} + \frac{1}{\beta}, & \hbox{for MINI-element.}  \end{array}\right.$$\\

With the values of $L$ given above, it is not possible to find a parameter $\gamma$ in~\eqref{splitting_gamma} so that the splitting has the form in~\eqref{splitting}, yielding an optimal iterative coupling method that converges in two iterations. However, we can include a second parameter into the iterative coupling method so that we retain the required splitting form in~\eqref{splitting}. To do so, we consider two parameters $\gamma_1, \gamma_2 >0$, and if  ${\bm u}_h^{n,0} = {\bm u}_h^{n-1}$ and  $p_h^{n,0} = p_h^{n-1}$ are the initial guesses of the iterative scheme at time $t_n$, the new improved algorithm provides a sequence of approximations $({\bm u}_h^{n,i},p_h^{n,i}), i \geq 1$ as follows: \\

\noindent {\bf Step 1:} Given $({\bm u}_h^{n,i-1},p_h^{n,i-1}) \in {\bm V}_h \times Q_h$, find $p_h^{n,i} \in Q_h$ such that:
\begin{eqnarray}  
  &&  \frac{1}{\beta} \left(\frac{p_h^{n,i}-p_h^{n-1}}{\tau},q_h\right) + \gamma_1 L \left(\frac{p_h^{n,i}-p_h^{n-1}}{\tau},q_h\right)_0 - \gamma_2 L \left(\frac{p_h^{n,i}-p_h^{n-1}}{\tau},q_h\right) + a_p(p_h^{n,i},q_h) = \nonumber \\
  & &\qquad\qquad + 
(1-\gamma_2) L \left(\frac{p_h^{n,i-1}-p_h^{n-1}}{\tau},q_h\right) + (\gamma_1-1) L \left(\frac{p_h^{n,i-1}-p_h^{n-1}}{\tau},q_h\right)_0 \nonumber \\ 
& & \qquad\qquad\qquad-\alpha \left(\ddiv \frac{\bm{u}_h^{n,i-1}-\bm{u}_h^{n-1}}{\tau},q_h\right) +  (g_h^n,q_h),
   \quad \forall \ q_h \in Q_h, \label{it_pressure3} 
\end{eqnarray}
{\bf Step 2:} Given $p_h^{n,i} \in Q_h$, find ${\bm u}_h^{n,i} \in {\bm V}_h$ such that
\begin{equation}
a(\bm{u}_h^{n,i},\bm{v}_h) =  \alpha( p_h^{n,i}, \ddiv \bm{v}_h) + (\bm{f}_h^n,\bm{v}_h),
     \quad \forall \  \bm{v}_h \in \bm V_h.  \label{it_displacement3}
\end{equation}

\noindent Including both parameters $\gamma_1$ and $\gamma_2$, the matrix form of the proposed iterative method is based on the following splitting of $\mathcal{A}_{stab}$,                               
\begin{equation}\label{splitting_gamma2}
\mathcal{A}_{stab} \!=\! 
\begin{pmatrix}\!
A& G\\
0 &  \tau A_p + \beta^{-1} M + \gamma_1 L M_l -  \gamma_2 L M 
\!\end{pmatrix}\;
\! \!- \! \begin{pmatrix}\!
0& 0\\
-D &  (\gamma_1-1) L M_l - (\gamma_2-1) L M 
\!\end{pmatrix}\!.
\end{equation}

Again, we prove the convergence of the new iterative method for a set of values for both parameters $\gamma_1$ and $\gamma_2$.
Since the constants arising from this analysis are independent of any physical and discretization parameters, we obtain an iterative coupling method which is parameter-robust. 
Let again $e_u^i = \bm{u}_h^{n,i} - \bm{u}_h^n$ and $e_p^i = p_h^{n,i}-p_h^n$ denote the errors at iteration $i$  for the displacements and for the pressure, respectively.\\

\begin{theorem} \label{thm:two-gammas}
	If the weak inf-sup condition~\eqref{ine:weak-inf-sup} holds, the iterative method given in \eqref{it_pressure3}-\eqref{it_displacement3} 
	is convergent for any parameters $\gamma_1 \in (1/2,2]$, $\gamma_1 > \gamma_2 \geq 0$, and $L$ such that  $(\gamma_1 -\gamma_2) L = \omega \frac{\alpha^2}{ (\lambda+2\mu/d)} \geq \frac{\alpha^2}{(\lambda+2\mu/d)}$, i.e., $\omega \geq 1$. Additionally, 
	\begin{equation}
		\label{thm_contraction_pressure_4}
		\| e_p^i \|^2  + \frac{|1-\gamma_1|}{\gamma_1-\gamma_2} \| e_p^i \|_Z^2   \leq  \frac{1}{ 1+  \frac{\eta^2}{ \omega (1 + 2  \theta^*) } } \left(  \| e_p^{i-1} \|^2 + \frac{|1-\gamma_1|}{\gamma_1-\gamma_2} \| e_p^{i-1} \|_Z^2  \right),
	\end{equation}
	where $\theta^* \geq \frac{\epsilon^2 C_2 \gamma_1}{4\omega (\gamma_1 - |1-\gamma_1|) }$ is a root of the quadratic equation~\eqref{def:q2} (with $\gamma$ replaced by $\gamma_1$).    Here, $\eta > 0$ is the constant appearing in the weak inf-sup condition~\eqref{ine:weak-inf-sup} and $C_2$ is the constant for the upper bound of the spectral equivalence condition~\eqref{ine:spec-equiv}.
\end{theorem}
\begin{proof}
By following the same procedure of the proof of Theorem~\ref{thm:robust-convergence-one-gamma}, using the fact that $\gamma_1 > \gamma_2$, if $(\gamma_1 - \gamma_2) L = \omega  \frac{\alpha^2}{\lambda + 2\mu/d} \geq \frac{\alpha^2}{\lambda + 2\mu/d}$, i.e., $\omega \geq 1$, we have
	\begin{align*}
		& \quad \left( \frac{1}{2} \frac{\eta^2}{1 + 2 \theta} \frac{(\gamma_1 - \gamma_2)}{\omega}  + \frac{(\gamma_1 - \gamma_2)}{2} \right) \| e_p^i \|^2 +\left( (1 - \frac{\epsilon^2 C_2}{4 \omega \theta}) \gamma_1 - \frac{|1-\gamma_1|}{2}  \right)  \| e_p^i \|_Z^2 \\
		& \leq \frac{(\gamma_1 - \gamma_2) }{2} \| e_p^{i-1} \|^2 + \frac{|1-\gamma_1|}{2} \| e_p^{i-1} \|_Z^2,
	\end{align*}
	where, again, $\theta>0$ is a constant when using Young's inequality.  To ensure convergence, it is sufficient that
	\begin{equation*}
		\left( (1 - \frac{\epsilon^2 C_2}{4 \omega \theta}) \gamma_1 - \frac{|1-\gamma_1|}{2}  \right) > \frac{|1-\gamma_1| }{2} \Longrightarrow \theta \geq \frac{\epsilon^2 C_2 \gamma_1}{4 \omega (\gamma_1 - |1-\gamma_1|)}.
	\end{equation*}
	Similar to Theorem~\ref{thm:robust-convergence-one-gamma}, we try to find $\theta = \theta^* \geq \frac{\epsilon^2 C_2 \gamma_1}{4 \omega (\gamma_1 - |1-\gamma_1|)}$ such that
	\begin{equation*}
		\frac{(1 - \frac{\epsilon^2 C_2}{4 \omega \theta}) \gamma_1 - \frac{|1-\gamma_1|}{2}}{\frac{1}{2} \frac{\eta^2}{1 + 2 \theta} \frac{(\gamma_1 - \gamma_2)}{\omega}  + \frac{(\gamma_1 - \gamma_2)}{2}} = \frac{|1-\gamma_1|}{\gamma_1 - \gamma_2}.
	\end{equation*}
	Direct calculations show that $\theta^* \geq  \frac{\epsilon^2 C_2 \gamma_1}{4 \omega (\gamma_1 - |1-\gamma_1|)}$ is a positive root of the quadratic equation~\eqref{def:q2} (with $\gamma$ replaced by $\gamma_1$).  The existence of $\theta^*$ follows by the same argument used in the proof of Theorem~\ref{thm:robust-convergence-one-gamma}. Thus, we have
	\begin{equation*}
		\left( \frac{\eta^2 }{2 \omega (1 + 2 \theta^*) } + \frac{1 }{2} \right) \left( \| e_p^i \|^2  + \frac{|1-\gamma_1|}{\gamma_1-\gamma_2} \| e_p^i \|_Z^2 \right)  \leq  \frac{1}{2} \left(  \| e_p^i \|^2 + \frac{|1-\gamma_1|}{\gamma_1-\gamma_2}  \| e_p^{i-1} \|_Z^2  \right),
	\end{equation*}
	which completes the proof.
	\end{proof}

In order to fix parameters $\gamma_1$ and $\gamma_2$ we consider again the one-dimensional Terzaghi's problem and use Theorem~\ref{thm:optimal}. As before, the values of these parameters are fixed depending on the considered discretization, yielding an optimal method at least for one-dimensional problems. 

\begin{corollary}\label{cor_1b}
Iterative method \eqref{it_pressure3}-\eqref{it_displacement3} with parameters  $L = \frac{3\alpha^2}{2(\lambda+2\mu)}+\frac{1}{\beta}$, $\gamma_1 = 1- \frac{1}{2L(\lambda+2\mu)}$, and $\gamma_2 = 1-\frac{3}{2L(\lambda+2\mu)}$ converges in two iterations for the stabilized P1-P1 discretization of Terzaghi's problem.
\end{corollary}

\begin{corollary}\label{cor_2b}
Iterative method \eqref{it_pressure3}-\eqref{it_displacement3} with parameters  $ L = \frac{\alpha^2}{\lambda+2\mu}+\frac{1}{\beta}$,  $\gamma_1= 1$, and $\gamma_2 =  \frac{1/\beta}{L}$ converges in two iterations for the stabilized MINI-element discretization of the Terzaghi's problem. 
\end{corollary}

The proofs of Corollary~\ref{cor_1b} and Corollary~\ref{cor_2b} are similar to those of Corollary~\ref{cor_1} and Corollary~\ref{cor_2}, respectively. The corresponding Schur complements $S_p$ are explicitly computed and choosing the given values of parameters $L$, $\gamma_1$, and $\gamma_2$ we observe that the splitting corresponding to the iterative method \eqref{it_pressure3}-\eqref{it_displacement3}  given in \eqref{splitting_gamma2} matches with the decomposition \eqref{splitting}  in Theorem~\ref{thm:optimal}, and, therefore, such schemes converge in only two iterations.

\section{Conclusions}\label{sec:conclusions}
In this work, we have dealt with the two crucial aspects appearing in the numerical simulation of Biot's model, namely the choice of an appropriate discretization approach which provides solutions without nonphysical oscillations and the efficient and robust solution of the resulting large linear system of equations. We have proposed a novel stabilization strategy which has been applied to two different discretizations of the problem: the P1-P1 and MINI-element schemes. Such a stabilization presents several advantages over others.  First, it does not depend on the mesh size $h$, or other discretization parameters. Second, the particular form of the stabilization results in an iterative coupling scheme between the fluid and mechanics problems, which is convergent without the need to introduce other stabilization terms. Thus, we obtain a convergent iterative solver for the obtained large system of equations, which we also demonstrate is parameter-robust. Moreover, this iterative algorithm can be tuned so that we optimize the convergence of the method for one-dimensional problems, and nearly-optimal convergence results in two and three dimensions. 
The proposed approach is presented in detail for the case in which the storage coefficient $1/\beta$ is close to zero, and its extension to a more general case which includes the regime with larger values of $1/\beta$ is demonstrated.  Future work involves applying these methods to other applications which present similar numerical difficulties.

\section*{Acknowledgments}
The work of \'Alvaro P\'e de la Riva, Francisco J. Gaspar and Carmen Rodrigo is supported in part by the Spanish project PID2022-140108NB-I00 (MCIU/AEI/FEDER, UE), and by the DGA (Grupo de referencia APEDIF, ref. E24\_17R). 
The work of Adler and Hu is partially supported by the National Science Foundation (NSF) under grant DMS-2208267. \\
The research of Ludmil Zikatanov 
is based upon work supported by and while serving at the National Science Foundation. 
Any opinion, findings, and conclusions or recommendations expressed in this material are those of the author 
and do not necessarily reflect the views of the National Science Foundation.

\bibliography{mybibfile2}

\end{document}